\documentclass[a4paper,11pt]{article}
\usepackage{graphicx}
\usepackage[T1]{fontenc}
\usepackage[utf8]{inputenc}
\usepackage[english]{babel}
\usepackage{csquotes}

\usepackage[p,osf]{scholax}
\usepackage{amsmath,amsthm}
\usepackage[scaled=1.075,ncf,vvarbb]{newtxmath}

\usepackage{subcaption}
\usepackage{geometry}
 \usepackage{microtype}
 \usepackage{nicefrac}
 \usepackage{empheq}
 
\usepackage[svgnames,x11names,table]{xcolor} 
 
\usepackage{pgfplots}
\usepackage{hyperref}
\hypersetup{
    colorlinks=true,        
    linkcolor=DodgerBlue,   
    citecolor=DarkRed,      
    urlcolor=Teal           
} 

\usepackage{framed}
\usepackage[noend]{algorithmic}
\usepackage{algorithm}
 
\usepackage{enumerate}
\usepackage{cleveref}

\newtheorem{theorem}{Theorem}
\theoremstyle{definition}
\newtheorem{lemma}{Lemma}

\newtheorem{remark}{Remark}

\usepackage{bbm}

\newcommand{\n}[1]{\|#1 \|}

\renewcommand{\a}{\alpha}

\renewcommand{\d}{\delta}

\newcommand{\la}{\lambda}

\renewcommand{\phi}{\varphi}

\renewcommand{\th}{\theta}

\newcommand{\x}{\bar x}

\newcommand{\R}{\mathbbm R}
\newcommand{\eR}{(-\infty,+\infty]}

\newcommand{\cO}{\mathcal O}
\newcommand{\cC}{\mathcal C}
\newcommand{\cX}{\mathcal X}

\newcommand{\lr}[1]{\langle #1\rangle}
\newcommand{\lrb}[1]{\left( #1 \right)}

\renewcommand{\leq}{\leqslant}
\renewcommand{\geq}{\geqslant}

\DeclareMathOperator{\prox}{prox}

\DeclareMathOperator{\dom}{dom}
\DeclareMathOperator{\id}{Id}

\DeclareMathOperator{\sign}{sign}

\DeclareMathOperator{\tr}{tr}
\def\<#1,#2>{\langle #1,#2\rangle}

\renewcommand{\ss}{\alpha}
\newcommand{\sbg}[1]{\widetilde\nabla g(x^{#1})}
\newcommand{\Sbg}[1]{\widetilde\nabla F(x^{#1})}
\newcommand{\opt}{\mapsto}

\usepackage[
style=alphabetic,backend=biber,
url=false,isbn=false,
sorting=nyt,maxbibnames=100,
giveninits=true, date=year
]{biblatex}
\AtBeginBibliography{\small}

\newcommand{\block}[3]{
 \hspace{-1.5mm}
 \begin{tikzpicture}
   [
   ]
   \node[fill={rgb,255:red,#1;green,#2;blue,#3},rectangle, minimum width=0.5cm,
                        minimum height = 0.15cm, inner sep=0pt, thick, rounded
                        corners=.09cm,
                        line width=0pt] () [] {};
 \end{tikzpicture}%
}

\graphicspath{ {plots/} }

\def\toptitlebar{\hrule height1pt \vskip .25in}
\def\bottomtitlebar{\vskip .22in \hrule height1pt \vskip .3in}

\title{\toptitlebar\textbf{Adaptive proximal gradient method for convex optimization}\bottomtitlebar}
\author{\textbf{Yura Malitsky}\footnote{Faculty of Mathematics, University of Vienna, Austria,
        \href{mailto:yurii.malitskyi@univie.ac.at}{yurii.malitskyi@univie.ac.at}}
    \and \textbf{Konstantin Mishchenko}\thanks{Samsung AI Center, UK. Work done prior to joining the company. \href{mailto:konsta.mish@gmail.com}{konsta.mish@gmail.com}}}
\date{}

\begin{document}
\maketitle
\begin{abstract}
  In this paper, we explore two fundamental first-order algorithms in convex optimization, namely, gradient descent (GD) and proximal gradient method (ProxGD). Our focus is on making these algorithms entirely adaptive by leveraging local curvature information of smooth functions. We propose adaptive versions of GD and ProxGD that are based on observed gradient differences and, thus, have no added computational costs. Moreover, we prove convergence of our methods assuming only \emph{local} Lipschitzness of the gradient. In addition, the proposed versions allow for even larger stepsizes than those initially suggested in~\cite{malitsky2019_descent}.  
\end{abstract}

\section{Intro}
In this paper, we address a convex minimization problem
\[\min_{x\in\R^d}F(x).\]
We are interested in the cases when either $F$ is differentiable and then we
will use notation $F = f$, or it has a composite additive structure as
$F = f + g$. Here, $f$ represents a convex and differentiable function, while
$g$ is convex, lower semi-continuous (lsc), and prox-friendly. Throughout the
paper, we will interchangeably refer to the smoothness of $f$ and the
Lipschitzness of $\nabla f$, occasionally with the adjective "locally,"
indicating that it is restricted to a bounded set. We will refer to this
property as smoothness, without mentioning the Lipschitzness of $f$, so we hope
there will be no confusion in this regard.

\bigskip

For simplicity, in most of the introduction, we consider only the simpler problem
$\min_x f(x)$. We study one of the most classical optimization algorithm~---
\emph{gradient descent}~---
\begin{equation}
  \label{eq:GD}
  x^{k+1} = x^k - \ss_k \nabla f(x^k).
\end{equation}
Its simplicity and the sole prerequisite of knowing  the gradient of $f$ make it appealing for diverse applications. This method is central in modern continuous optimization, forming the bedrock for numerous extensions.

Given the initial point $x^0$, the only thing we need to implement \eqref{eq:GD}
is to choose a stepsize $\ss_k$ (also known as a learning rate in machine
learning literature). This seemingly tiny detail is crucial for the method
convergence and performance. When a user invokes GD as a solver, the standard
approach would be to pick an arbitrary value for $\ss_k$, run the algorithm, and
observe its behavior. If it diverges at some point, the user would try a smaller
stepsize and repeat the same procedure. If, on the other hand, the method takes
too much time to converge, the user might try to increase the stepsize. In
practice, this approach is not very efficient, as we have no theoretical
guarantees for a randomly guessed stepsize, and the divergence may occur after a
long time. Both underestimating and overestimating the stepsize can, thus, lead
to a large overhead. \bigskip

Below we briefly list possible approaches to choosing or estimating the stepsize and we provide a more detailed literature overview in Section~\ref{sec:literature}.

\paragraph{Fixed stepsize.} When $f$ is $L$-smooth,  GD can utilize a fixed stepsize $\ss_k=\ss < \frac 2 L$ and values larger than $\frac{2}{L}$ will provably lead to divergence. Consequently, in such scenarios, the rate of convergence is given by $f(x^k)-f_* = \cO\left(\frac{1}{\ss k}\right)$, clearly indicating a direct dependence on the stepsize. Nevertheless, several drawbacks emerge from this approach:
\begin{enumerate}[(a)]
\item $L$ is not available in many practical scenarios;
\item if the curvature of $f$ changes a lot, GD with the global value of $L$ may be too conservative;
\item $f$ may be not globally $L$-smooth.
\end{enumerate}
For illustration, consider the following functions. Firstly, when dealing with
$f(x) = \frac{1}{2}\|Ax-b\|^2$, where $A\in \R^{n\times d}$ and $b\in\R^n$,
estimating $L$ involves evaluating  the largest eigenvalue of $A^\top A$.
Second, the logistic loss $f(x)=\log(1+\exp(-ba^\top x))$, with $a\in\R^d, b\in
\{-1,1\}$, is almost flat for large $x$, yet for values of $x$ closer to $0$, it
has $L=\frac{1}{4}\|a\|^2$ . Thus, if the solution is far from $0$, gradient
descent with a constant stepsize would be too conservative. Finally, consider $f(x) = x^4$. While this simple objective is not globally $L$-smooth for any value of $L$, on any bounded set it \emph{is} smooth, and we would hope we can still minimize objectives like that.
  
\paragraph{Linesearch.} Also known as backtracking in the literature.  In the $k$-th iteration we compute $x^{k+1}$ with a certain stepsize $\ss_k$ and check a specific condition. If the condition holds, we accept $x^{k+1}$ and proceed to the next iteration; otherwise we halve $\a_k$ and recompute $x^{k+1}$ using this reduced stepsize. This approach, while the most robust and theoretically sound, incurs substantially higher computational cost compared to regular GD due to the linesearch procedure.

\paragraph{Adagrad-type algorithms.} These are the methods of the
type\footnote{We provide only the simplest instance of such algorithms.}
\begin{equation}
  \label{eq:adagrad}
  \begin{aligned}
    v_{k} &= v_{k-1} + \n{\nabla f(x^k)}^2\\
    x^{k+1} &= x^k - \frac{d_k}{\sqrt{v_{k}}}\nabla f(x^k),
  \end{aligned}
\end{equation}
where $v_{-1} \ge 0$ is some constants, and $d_k$ is an estimate of $\|x^0 - x^*\|$ some some solution $x^*$. 
While such methods indeed have certain nice properties, $d_k$ is usually either constant or quickly converges to a constant value, so a quick glance at
\eqref{eq:adagrad} will reveal that its stepsizes are decreasing. Therefore, despite the name, we
cannot expect true adaptivity of this method to the local curvature of $f$.

\paragraph{Heuristics.}
Numerous heuristics exist for selecting $\ss_k$ based on local properties of $f$ and $\nabla f$, with the Barzilai-Borwein method \cite{barzilai1988two} being among the most widely popular. However, it is crucial to note that we are not particularly interested in such approaches, as they lack consistency and may even lead to divergence, even for  simple convex problems.

\bigskip

We have already mentioned \emph{adaptivity} a few times, without properly introducing it. Now let us try to properly understand its meaning in the context of gradient descent. Besides the initial point $x^0$, GD has only one degree of freedom~--- its stepsize. From the analysis we know that it has to be approximately an inverse of the local smoothness.  We call a method \emph{adaptive}, if it automatically adapts a stepsize to this local smoothness without additional expensive computation and the method does not deteriorate the rate of the original method in the worst case. In our case, original method is GD with a fixed stepsize.

By this definition, GD with linesearch is not adaptive, because it finds the right stepsize with some extra evaluations of $f$ or $\nabla f$. GD with diminishing steps (as in subgradient or Adagrad methods) is also not adaptive, because decreasing steps cannot in general represent well the function's curvature; also the rate of the subgradient method is definitely worse.  It goes without saying, that for a \emph{good} method its rate must experience improvement when we confine the class of smooth convex functions to the strongly convex ones. 

\begin{center}
  \textsc{Contribution}
\end{center}

In a previous work \cite{malitsky2019_descent}, which serves as the cornerstone for the current paper, the authors
proposed an adaptive gradient method named \emph{``Adaptive Gradient
Descent without Descent''} (AdGD). In the current paper, we
\begin{itemize}
\item  deepen our understanding of AdGD and identify its limitations;
\item  refine its theory to accommodate  even larger steps;
\item  extend the revised algorithm from  unconstrained to the proximal case.
\end{itemize}
The analysis in the last two cases is not a trivial extension, and we were
rather pleasantly  surprised that this was possible at all. After all, the theory of GD 
is well-established and we thought it to be too well-explored for us to discover something new.

\paragraph{Continuous point of view.} It is instructive for some time to switch from the discrete setting to the
continuous and to compare gradient descent (GD) with its parent~--- gradient
flow (GF)
\begin{equation}
  \label{eq:gf}
  x'(t) = -\nabla f(x(t)), \qquad x(0)=x_0,
\end{equation}
where $t$ is the time variable and $x'(t)$ denotes the derivative of $x(t)$ with respect to $t$. 
To guarantee existence and uniqueness of a trajectory $x(t)$ of GF, it is
sufficient to assume that $\nabla f$ is \emph{locally} Lipschitz-continuous.
Then one can prove convergence of $x(t)$ to a minimizer of $f$ in just a few
lines. For GD, on the other hand, the central assumption is \emph{global}
Lipschitzness of $\nabla f$. Our analysis of gradient descent makes it level:
local Lipschitzness suffices for both. Or to put it differently, we provide an
adaptive discretization of GF that converges under the same assumptions as the
original continuous problem~\eqref{eq:gf}.

\paragraph{Proximal case.}
We emphasize that there is already an excellent extension by Latafat et al.~\cite{latafat2023adaptive} of the work~\cite{malitsky2019_descent} to the additive composite case. Our proposed result, however, is based on an improved unconstrained analysis and uses a different proof. We believe that both these facts will be of interest. We don't have a good understanding why, but for us finding the proof for the proximal case was quite challenging. It does not follow the standard lines of arguments and uses a novel Lyapunov energy  in the analysis.

\paragraph{Nonconvex problems.} We believe that our algorithm will be no less important in the nonconvex case, where gradients are rarely globally Lipschitz continuous and where the curvature may change more drastically. It is true that our analysis applies only to the convex case, but, as far as we know, limited theory has never yet prevented practitioners from using methods in a broader setting. And based on our (speculative) experience, we found it challenging to identify nonconvex functions where the method did not converge to a local solution.

\paragraph{Outline.} In \Cref{sec:improved}, we begin by revisiting AdGD from
\cite{malitsky2019_descent}, examining its limitations, and demonstrating a
simple way to enhance it. This section maintains an informal tone, making it
easily accessible for quick reading and classroom presentation. In
\Cref{sec:larger}, we further improve the method and provide all formal proofs.
\Cref{sec:prox} extends the improved method to the proximal case. In
\Cref{sec:literature} we put our finding in the perspective and compare it to
some existing works. Lastly, in \Cref{sec:exp}, we conduct experiments to
evaluate the proposed method against different linesearch variants.

 \subsection{Preliminaries}
We say that a mapping is \emph{locally Lipschitz} if it is Lipschitz over any
compact set of its domain.  A function $f\colon \R^d\to \R$ is \emph{(locally) smooth} if its gradient $\nabla f$
is  (locally) Lipschitz.

A convex $L$-smooth function $f$ is characterized by the following inequality
\begin{equation}
  \label{eq:Lsmooth}
  f(y)-f(x) - \lr{\nabla f(x), y-x} \geq \frac{1}{2L} \n{\nabla f(y)-\nabla f(x)}^2 \quad \forall x,y.
\end{equation}
This is equivalently of saying  that $\nabla f$ is a $\frac{1}{L}$-\emph{cocoercive}
operator, that is
\begin{equation}
  \label{eq:coco}
  \lr{\nabla f(y) - \nabla f(x), y-x} \geq \frac{1}{L} \n{\nabla f(y)-\nabla f(x)}^2 \quad \forall x,y.
\end{equation}
For a convex differentiable $f$ that is not $L$-smooth one can only say that $\nabla f$ is \emph{monotone},
that is
\begin{equation}
  \label{eq:mono}
  \lr{\nabla f(y) - \nabla f(x), y-x} \geq 0\quad \forall x,y.
\end{equation}
We  use notation $[t]_+ = \max\{t, 0\}$ and for any $a>0$ we suppose that
$\frac{a}{0}=+\infty$. With a slight abuse of notation, we write $[n]$ to denote the
set $\{1,\dots, n\}$.
A solution and the value of the optimization problem
$\min f(x)$ are denoted by $x^*$ and $f_*$, respectively.

\section{Adaptive gradient descent: better analysis}\label{sec:improved}
Let us start with the simpler problem of $\min_{x} f(x)$ with a convex, locally smooth
$f\colon \R^{d}\to \R$. To solve it, in \cite{malitsky2019_descent}, the authors proposed a method called \emph{adaptive gradient descent without descent} (AdGD), whose update is given below:
\begin{equation}\label{eq:old_adgd}
    \begin{aligned}
        \ss_k &= \min\Bigl\{
        \sqrt{1+\th_{k-1}}\ss_{k-1}, \frac{\n{x^k-x^{k-1}}}{2\n{\nabla f(x^k)-\nabla f(x^{k-1})}}\Bigr\}, \qquad \textrm{where }\th_k = \frac{\ss_k}{\ss_{k-1}} \\
        x^{k+1} & = x^k - \ss_k \nabla f(x^k).
    \end{aligned}
  \end{equation}
Similarly to the standard GD, this method leads to $\cO(1/k)$ convergence rate.
However, unlike the former, it doesn't require any knowledge about Lipschitz
constant of $\nabla f$ and doesn't even require a global Lipschitz continuity of
$\nabla f$.

The update for $\ss_{k}$ has two ingredients. The first bound
$\ss_{k}\leq \sqrt{1+\th_{k-1}}\ss_{k-1}$ sets how fast steps may increase from
iteration to iteration. The second
$\ss_{k}\leq \frac{\n{x^k-x^{k-1}}}{2\n{\nabla f(x^k)-\nabla f(x^{k-1})}}$
corresponds to the estimate of local Lipschitzness of $\nabla f$.

It is important to understand how essential these bounds are. Do we
really need to control the growth rate of $\ss_{k}$ or is it an artifact of our
analysis? For the second bound, it is not clear whether $2$ in the denominator
is necessary. For example, given $L$-smooth $f$, our scheme \eqref{eq:old_adgd}
does not encompass a standard GD with $\ss_{k}=\frac{1}{L}$ for all $k$. 

\paragraph{First bound.} Answering the first question is relatively easy. Consider the following function

\begin{minipage}{.49\textwidth}
  \begin{align}
    \hspace{-1cm} f(x) =
        \begin{cases}
                \frac{1}{2}x^2,& x\in [-1, 1] \\
                a(|x|-\log(1+|x|))+b, & x\not\in [-1, 1]
        \end{cases}
  \label{eq:counter} 
\end{align}
\end{minipage} 
\qquad \qquad \qquad \begin{minipage}{.4\textwidth}
\begin{tikzpicture}
    \def\aa{2}
  \def\bb{2*ln(2)-1.5}
\begin{axis}[
 axis lines=center,
  axis line style={thick},
  x axis line style={-stealth},
  xmin=-7.3, xmax=7.3,
  ymin=-0.3, ymax=10.3,
  xtick distance=1,
  ytick distance=1,
  xticklabels={},
  yticklabels={},
  width=5cm, height=3cm,
  scale only axis,
  xlabel={$x$},
  ylabel={$f(x)$},
  xlabel style={right},
  ylabel style={right},  
  xtick={1},
  ytick={1},
  xticklabels={\tiny{$1$}},
  yticklabels={\tiny{$1$}}
  ]
  \addplot[domain=-1:1,cyan!80!black, ultra thick]              {0.5 * \x * \x}  ;
  \addplot[domain=-7:-1,cyan!80!black, ultra thick]        {\aa * (-\x - ln(1 -    \x)) + \bb}       ;
  \addplot[domain=1:7,cyan!80!black, ultra thick]        {\aa * (\x - ln(1 + \x)) + \bb}       ;

\end{axis}
\end{tikzpicture}

\end{minipage}
\\where parameters $a, b>0$ are chosen to ensure that $f(\pm 1)$ and
$f'(\pm 1)$ are well-defined, namely  $a=2$ and $b=2\log
2-\frac{3}{2}$, see \Cref{prop:bad_f} in \hyperlink{appendix}{Appendix}.

From an optimization point of view, $f$ is a nice function. In particular, it is
convex (even locally strongly convex) and its gradient is $1$-Lipschitz,
see~\Cref{prop:bad_f}. This means that both GD and AdGD linearly converge on
it. However, if we remove the first condition for $\ss_k$ in AdGD, this new
modified algorithm will fail to converge. We can prove an even stronger
statement. Specifically, let $c\geq 1$, $\ss_0=1$ and consider the following
method
\begin{equation}\label{bad_gd}
  \begin{aligned}
    \ss_k &= \frac{\n{x^k-x^{k-1}}}{c\n{\nabla f(x^k) - \nabla f(x^{k-1})}}, \quad \forall k\geq 1\\
  x^{k+1} &= x^k - \ss_k \nabla f(x^k), \quad \forall k\geq 0.
  \end{aligned}
\end{equation}
In other words, the update in~\eqref{bad_gd} is the same as in \eqref{eq:old_adgd} except 
we removed the first constraint for $\ss_k$  in \eqref{eq:old_adgd} and introduced a constant 
factor $c$ to make the second one more general.
\begin{theorem}\label{prop:counter}
For any $c\geq 1$ there exists $x^0$ such that the method \eqref{bad_gd} applied to
$f$ defined in~\eqref{eq:counter}  diverges. 
\end{theorem}   
The formal proof of this statement is in \hyperlink{appendix}{Appendix}, but
its main idea should be intuitively clear. First, observe that for $x$ with
large absolute value, $f(x)$ behaves mostly like a linear function. However,
$f'(x)$ approaches $-1$ when $x\to -\infty $ and $+1$ when $x\to+\infty$.
Therefore, if $x^k$ and $x^{k-1}$ have the same sign, the local smoothness
estimate will be too optimistic and $x^{k+1}$ will ``leapfrog'' the optimum. In
contrast, if the signs of $x^k$ and $x^{k-1}$ are different, then $x^{k+1}$ will
fail to get sufficiently close to the optimum. It is interesting to remark that on this function both versions of the Barzilai-Borwein method will diverge as well.
\bigskip

Consequently, the answer to the first question is affirmative: we do need some extra condition for the stepsize $\ss_k$.

\paragraph{Second bound.}
The answer to the second question is the opposite: it is indeed an artifact of our previous
analysis. In the next section, we propose an improvement over the previous
version \cite{malitsky2019_descent}. We give a concise presentation in an informal way. 
We keep a more formal style for \cref{sec:larger} where even better version
(also slightly more complicated) will be presented.

\subsection{Improving AdGD}
The analysis of GD usually starts from the standard identity, followed by convexity inequality
\begin{align}
  \label{eq:adgd_main1}
  \n{x^{k+1}-x^{*}}^{2} &= \n{x^{k}-\ss_{k}\nabla f(x^{k}) - x^{*}}^{2} \notag \\
                        &= \n{x^{k}-x^{*}}^{2} - 2\ss_{k}  \lr{\nabla f(x^{k}),x^{k}-x^{*}} + \ss_{k}^{2}\n{\nabla f(x^{k})}^{2}\notag \\
                        &\leq \n{x^{k}-x^{*}}^{2} - 2\ss_{k} \bigl(f(x^{k})-f(x^{*})\bigr) + \ss_{k}^{2}\n{\nabla f(x^{k})}^{2}.
\end{align}

In~\cite{malitsky2019_descent} the only ``nontrivial'' step in the proof was upper bounding
$\ss_k^2\n{\nabla f(x^k)}^2$, that is $\n{x^{k+1}-x^{k}}^{2}$. Now we
do it in a slightly different way. First, we need the following fact.
\begin{lemma}\label{prop1}
  For GD iterates $(x^{k})$ with arbitrary stepsizes, it holds
\begin{equation}\label{eq:prop1}
  \lr{\nabla f(x^{k}), \nabla f(x^{k-1})} \leq \n{\nabla f(x^{k-1})}^{2}.
\end{equation}
\end{lemma}
\begin{proof}
  This is just monotonicity of $\nabla f$ in disguise:
  \begin{align*}
    \n{\nabla f(x^{k-1})}^{2} - \lr{\nabla f(x^{k}), \nabla f(x^{k-1})}
    &= \lr{\nabla f(x^{k-1})-\nabla f(x^{k}), \nabla f(x^{k-1})} \\
    &=  \frac{1}{\ss_{k-1}}\lr{\nabla f(x^{k-1})-\nabla f(x^{k}), x^{k-1}-x^{k}}\geq 0.\qedhere
  \end{align*}
\end{proof}

Now we are going to bound $\n{x^{k+1}-x^k}^2$. For convenience, denote the
approximate local Lipschitz constant as
\[L_k = \frac{\n{\nabla f(x^{k})-\nabla f(x^{k-1})}}{\n{x^{k}-x^{k-1}}}.\] 
Let $\ss_{k}$ satisfy
$\ss_{k}\n{\nabla f(x^{k})-\nabla f(x^{k-1})}\leq \gamma \n{x^{k}-x^{k-1}}$ for
some $\gamma >0 $, that is $\ss_k L_k \leq \gamma$. Using
$\n{u}^2 = \n{u-v}^2 - \n{v}^2 + 2\lr{u,v}$, we have
\begin{align}\label{eq:decrease}
  \n{x^{k+1}-x^{k}}^{2} &= \ss_{k}^{2}\n{\nabla f(x^{k})}^{2} \notag\\
                        &= \ss_{k}^{2}\n{\nabla f(x^{k})-\nabla f(x^{k-1})}^{2} -\ss_{k}^{2}\n{\nabla f(x^{k-1})}^{2} + 2\ss_{k}^{2}\lr{\nabla f(x^{k}), \nabla f(x^{k-1})}\notag \\
                        &= \ss_k^2 L_k^{2} \n{x^{k}-x^{k-1}}^{2}  -\ss_{k}^{2}\n{\nabla f(x^{k-1})}^{2} + 2\ss_{k}^{2}\lr{\nabla f(x^{k}), \nabla f(x^{k-1})} \notag\\
                        &\overset{\eqref{eq:prop1}}{\leq} \gamma^{2} \n{x^{k}-x^{k-1}}^{2}  + \ss_{k}^{2}\lr{\nabla f(x^{k}), \nabla f(x^{k-1})}  \notag\\
                        &= \gamma^{2} \n{x^{k}-x^{k-1}}^{2}  + \ss_{k}\th_{k}\lr{\nabla f(x^{k}), x^{k-1}-x^{k}} \notag\\
                        &\leq \gamma^{2} \n{x^{k}-x^{k-1}}^{2}  + \ss_{k}\th_{k} (f(x^{k-1})-f(x^{k})),
\end{align}
where the last inequality follows from convexity of $f$. For $\gamma<1$ we can rewrite~\eqref{eq:decrease}  as
\begin{align*}
  \n{x^{k+1}-x^{k}}^{2} \leq \frac{\gamma^{2}}{1-\gamma^{2}} \n{x^{k}-x^{k-1}}^{2} -  \frac{\gamma^{2}}{1-\gamma^{2}} \n{x^{k+1}-x^{k}}^{2}   + \frac{\ss_{k}\th_{k}}{1-\gamma^{2}} (f(x^{k-1})-f(x^{k})).
\end{align*}
Substituting this inequality into \eqref{eq:adgd_main1} gives us
\begin{align}\label{eq:to_telescope}
  &\n{x^{k+1}-x^*}^2+ \frac{\gamma^{2}}{1-\gamma^{2}}\n{x^{k+1}-x^k}^2 +
  \ss_k\left(2 + \frac{\theta_k}{1-\gamma^2}\right)(f(x^{k})-f_{*}) \nonumber
  \\ \leq \, &\n{x^k-x^*}^2  +
  \frac{\gamma^{2}}{1-\gamma^{2}}\n{x^{k}-x^{k-1}}^2 + \frac{\ss_k\theta_k}{1-\gamma^2}(f(x^{k-1})-f_{*}).
\end{align}
As we want to telescope the above inequality, we require
\begin{align*}
  \frac{\ss_{k}\th_{k}}{1-\gamma^{2}}\leq \ss_{k-1}\left(2 +
    \frac{\th_{k-1}}{1-\gamma^{2}}\right) \iff
  \ss_{k}^2\leq \left(2(1-\gamma^2) + \th_{k-1}\right)\ss_{k-1}^2.
\end{align*}
On the other hand, we have already used that  $  \ss_{k} L_k \leq
\gamma$. These two conditions lead to the bound
\begin{equation}
  \label{eq:alpha_bound}
  \ss_k = \min\left\{ \sqrt{2(1-\gamma^2) + \th_{k-1}}\ss_{k-1}, \frac{\gamma}{L_k} \right\},
\end{equation}
where $\gamma \in (0,1)$ can be arbitrary.
Now by playing with different values of $\gamma$, we obtain
different instances of adaptive gradient descent method. For instance, by setting
$\gamma = \frac{1}{\sqrt 2}$, we get
\[\ss_k = \min\left\{\sqrt{1+ \th_{k-1}}\ss_{k-1}, \frac{1}{\sqrt 2 L_k}\right\},\]
which is a strict improvement upon the original version
in~\cite{malitsky2019_descent}. A simple reason why this is possible
is that, unlike in \cite{malitsky2019_descent}, we did not resort to
the Cauchy-Schwarz inequality and instead relied on
transformation~\eqref{eq:decrease} and
\Cref{prop1}.

\begin{minipage}[t]{0.43\textwidth}
 \begin{algorithm}[H]
 \caption{Adaptive gradient descent}
 \label{alg:unconstr}
 \begin{algorithmic}[1]
   \STATE \textbf{Input:} $x^0 \in \R^d$, $\th_{0}= 0$, $\ss_0>0$
   \STATE $x^1 =x^0 -\ss_0 \nabla f(x^0)$ 
   \FOR{$k = 1,2,\dots$}
   \STATE $L_k = \dfrac{\n{\nabla f(x^{k})-\nabla f(x^{k-1})}}{\n{x^{k}-x^{k-1}}}$
   \STATE
   $\ss_k = \min\left\{\sqrt{1+\theta_{k-1}}\ss_{k-1}, \dfrac{1}{\sqrt 2 L_k}\right\}$
   \STATE $x^{k+1} = x^k - \ss_k \nabla f(x^k)$
   \STATE $\theta_k = \frac{\ss_k}{\ss_{k-1}}$
    \ENDFOR
 \end{algorithmic}
\end{algorithm}
\end{minipage}
\hfill
\begin{minipage}[t]{0.49\textwidth}
 \begin{algorithm}[H]
 \caption{Adaptive gradient descent-2}
 \label{alg:unconstr-2}
 \begin{algorithmic}[1]
   \STATE \textbf{Input:} $x^0 \in \R^d$, $\th_{0}=\frac 1 3$, $\ss_0>0$
   \STATE $x^1 =x^0 -\ss_0 \nabla f(x^0)$ 
   \FOR{$k = 1,2,\dots$}
   \STATE $L_k = \dfrac{\n{\nabla f(x^{k})-\nabla f(x^{k-1})}}{\n{x^{k}-x^{k-1}}}$
  \STATE
  $\ss_k = \min\left\{ \sqrt{\frac{2}{3} + \theta_{k-1}}\ss_{k-1}, \frac{\ss_{k-1}}{\sqrt{[2\ss_{k-1}^2L_k^2-1]_{+}}}\right\}$
   \STATE $x^{k+1} = x^k - \ss_k \nabla f(x^k)$
   \STATE $\theta_k = \frac{\ss_k}{\ss_{k-1}}$
   \ENDFOR
 \end{algorithmic}
\end{algorithm}
\end{minipage}
\bigskip

We summarize the new scheme in~\Cref{alg:unconstr}.  We do not provide a formal proof for this scheme and hope that inequality~\eqref{eq:to_telescope} should be sufficient for the curious reader to complete the proof. In any case, the next section will contain a further improvement with all the missing proofs.

 \begin{remark}
One might notice that we have used several times monotonicity of
   $\nabla f$, where we actually could use a stronger property of
   cocoercivity~\eqref{eq:coco}. That is true, but we just prefer simplicity. We recommend
   work~\cite{latafat2023adaptive} that exploits cocoercivity in this
   framework.
 \end{remark}

 \section{Adaptive gradient descent: larger stepsize} \label{sec:larger}
 In this section we modify \Cref{alg:unconstr} to use even larger steps 
 resulting in \Cref{alg:unconstr-2}. This, however, will require slightly more
 complex analysis.

Recall the notation $[t]_+ = \max\{t,0\}$ and note that the second bound $\ss_k\leq \frac{\ss_{k-1}}{\sqrt{[2\ss_{k-1}^2L_k^2-1]_{+}}}$ is equivalent to
\begin{equation}
  \label{eq:bound2}
  \ss_k^2 L_k^2 - \frac{\ss_{k}^2}{2\ss_{k-1}^2}\leq \frac 1 2,
\end{equation}
which obviously allows for larger range of $\ss_k$ than
$\ss_k^2L_k^2\leq \frac 1 2$ in \Cref{alg:unconstr}. On the other hand, the
first bound $\ss_k \leq  \sqrt{\frac{2}{3} + \theta_{k-1}}\ss_{k-1}$ is
definitely worse. At the moment, it is not even clear whether it allows $\ss_k$
to increase. 
\begin{remark}
  A notable distinction between \Cref{alg:unconstr-2} and \Cref{alg:unconstr} is
  that the former allows to use a standard fixed step $\ss_k = \frac{1}{L}$,
  provided that $f$ is $L$-smooth. For instance, if we start from
  $\ss_0 = \frac 1 L$ and use $L\geq L_k$ in every iteration (we can always use
  a larger value), then it follows from~\eqref{eq:bound2} and
  $\th_{k-1}\geq \frac 1 3$ that $\ss_k=\frac 1L $ for all $k\geq 1$.
\end{remark}
\paragraph{Initial stepsize.} This is an important though not very exciting topic.
\Cref{alg:unconstr-2} requires an initial  stepsize $\ss_0$.  While, as it will
be proved later, the algorithm converges for any value
$\ss_0>0$ with  the rate
\[\min_{i\in [k]}(f(x^i)-f_*)\leq \frac{R^2}{2\sum_{i=1}^k\ss_i},\]
(see eq.~\eqref{eq:R} for the definition of $R$), the choice of $\ss_0$ will impact further steps due to the bound
$\ss_{k}\leq \sqrt{2/3 + \th_{k-1}}\ss_{k-1}$. Because of this
reason, we do not want to choose $\ss_0$ too small. On the other hand,
too large $\ss_0$ will make $R$ large.  To counterbalance these two
extremes, we suggest to do the following:
\begin{equation}
  \label{eq:ss_0}
    \text{choose\,  } \ss_0 \text{\,  such that \, }  \ss_0 L_1 \in \left[\frac{1}{\sqrt 2}, 2\right].
\end{equation}
The upper bound ensures that $\ss_0$ is not too large, while the lower ensures
that it is not too small either.  In most scenarios, this requires to run a
linesearch, but we emphasize one more time: it is only needed for the first
iteration. In some sense, our condition~\eqref{eq:ss_0} is similar to classical
Goldstein's rule \cite{goldstein1962cauchy} on selecting the stepsize: not too
small and not too big.

Of course, if we start with a very small $\ss_0$, only the first bound for $\ss_k$ will be active for some time, and we will eventually reach a reasonable range for a stepsize. However, linesearch with a more aggressive factor (say, $10$) will allow us to reach this range faster. If we start with $\ss_0=10^{-8}$ when in fact a reasonable range for steps in this region is $[1,10]$, then we will need at least $100$ iterations of our method, while linesearch with a factor $10$ will find it in less than $10$ iterations.

It may happen that the problem is degenerated in a sense that for any $\ss_0$, $\ss_0 L_1 < \frac{1}{\sqrt 2}$. In other words, increasing $\ss_0$ leads to decreasing $L_1$ and linesearch may never stop. In this case we should terminate a linesearch after $\ss_0$ reaches any prescribed value, say $1$.
\subsection{Analysis}
\begin{lemma}\label{lem:decr3}
  For  iterates $(x^k)$ of \Cref{alg:unconstr-2} it holds
\begin{align}\label{eq:decrease3}
  \n{x^{k+1}-x^{k}}^{2} \leq \frac 1 2\n{x^{k}-x^{k-1}}^{2} + \frac 32 \ss_{k}\th_{k}(f(x^{k-1})-f(x^k)).
\end{align}
 \end{lemma}
Before we continue, let us give some intuition for this lemma. Its
analysis follows mostly the same steps as
in~\eqref{eq:decrease}. However, now we will split $\ss_{k}^2\n{\nabla
  f(x^{k-1})}^2$ into two parts and use one of it to improve the
smoothness bound for $\ss_k$.
\begin{proof}
  We start from the third line in \eqref{eq:decrease} and then  apply  the above-mentioned splitting:
\begin{align}
     \n{x^{k+1}-x^{k}}^{2} =\, & \ss_k^2 L_k^{2} \n{x^{k}-x^{k-1}}^{2}  -\ss_{k}^{2}\n{\nabla f(x^{k-1})}^{2} + 2\ss_{k}^{2}\lr{\nabla f(x^{k}), \nabla f(x^{k-1})} \notag\\
                         =\, & \left(\ss_k^2 L_k^{2}-\frac{\ss_k^2}{2\ss_{k-1}^2}\right) \n{x^{k}-x^{k-1}}^{2} - \frac{\ss_{k}^{2}}{2}\n{\nabla f(x^{k-1})}^{2}  + 2\ss_{k}^{2}\lr{\nabla f(x^{k}), \nabla f(x^{k-1})}\notag\\
                           \overset{\eqref{eq:bound2}\&
                          \eqref{eq:prop1}}{\leq}& \frac 12
                        \n{x^{k}-x^{k-1}}^{2}  + \frac
                        32\ss_{k}^2\lr{\nabla f(x^{k}), \nabla f(x^{k-1})}\notag\\
                        =\, & \frac 12 \n{x^{k}-x^{k-1}}^{2}  + \frac 32\ss_{k}\th_{k}\lr{\nabla f(x^{k}), x^{k-1}-x^{k}}.
\end{align}
Convexity of $f$ completes the proof.
\end{proof}
\begin{lemma}\label{lem:unconstr2_energy}
  For  iterates $(x^k)$ of \Cref{alg:unconstr-2} and any solution $x^*$ it holds
\begin{align}\label{lemma_unconstr2_energy}
  &\n{x^{k+1}-x^*}^2+ \n{x^{k+1}-x^k}^2 + \ss_k(2+3\th_k)(f(x^{k})-f_{*})\notag   \\
  \leq \, &\n{x^k-x^*}^2  + \n{x^{k}-x^{k-1}}^2 + 3\ss_{k}\th_{k}(f(x^{k-1})-f_*).
\end{align}

\end{lemma}
\begin{proof}
  From \eqref{eq:decrease3} we have
  \begin{align}\label{eq:decrease3_reform}
  \n{x^{k+1}-x^{k}}^{2} \leq \n{x^{k}-x^{k-1}}^{2}-\n{x^{k+1}-x^k}^2  + 3\ss_{k}\th_{k}(f(x^{k-1})-f(x^k)).
  \end{align}
  Using this inequality in \eqref{eq:adgd_main1}, we get
\begin{align*}
  &\n{x^{k+1}-x^*}^2+ \n{x^{k+1}-x^k}^2 + 2\ss_k(f(x^{k})-f_{*})
   \\ \leq \, &\n{x^k-x^*}^2  + \n{x^{k}-x^{k-1}}^2 + 3\ss_{k}\th_{k}(f(x^{k-1})-f(x^k)),
\end{align*}
which is equivalent to \eqref{lemma_unconstr2_energy}.
\end{proof}

\begin{lemma}\label{lem2:bounded}
  The sequence $(x^k)$ is bounded. In particular, for any solution
  $x^*$ we have  $x^k\in B(x^*, R)$,
  where
  \begin{equation}
    \label{eq:R}
  R^2 = \n{x^0-x^*}^2  + 2\ss_0^2\n{\nabla f(x^0)}^2+\ss_{0}(f(x^{0})-f_*).
  \end{equation}
\end{lemma}
\begin{proof}
  The first bound for $\ss_k$ in \Cref{alg:unconstr-2} gives us $3\ss_{k}\theta_{k}\leq (2 + 3\theta_{k-1})\ss_{k-1}$. We use it in~\eqref{lemma_unconstr2_energy} and telescope then to obtain
  \begin{align}
    \label{eq:asdaf}
      &\n{x^{k+1}-x^*}^2+ \n{x^{k+1}-x^k}^2 + \ss_k(2+3\th_k)(f(x^{k})-f_{*})\notag   \\
  \leq \, &\n{x^1-x^*}^2  + \n{x^{1}-x^{0}}^2 + \ss_0(2+3\th_0)(f(x^{0})-f_*).
  \end{align}
  This immediately implies that $(x^k)$ is bounded, but we would like to obtain
  the bound without an intermediate iterate $x^1$. From~\eqref{eq:adgd_main1} we know that
  \[\n{x^1-x^*}\leq \n{x^0-x^*}^2  + \ss_0^2\n{\nabla f(x^0)}^2 -2\ss_{0}(f(x^{0})-f_*). \]
Combining it with~\eqref{eq:asdaf}, we deduce
\begin{align*}
      &\n{x^{k+1}-x^*}^2+ \n{x^{k+1}-x^k}^2 + \ss_k(2+3\th_k)(f(x^{k})-f_{*}) \\
  \leq \, &\n{x^0-x^*}^2  + 2\ss_0^2\n{\nabla f(x^0)}^2
  +3\th_{0}\ss_0(f(x^{0})-f_*).
  \end{align*}
Using that $\theta_0=\frac 1 3$ completes the proof.
\end{proof}
\begin{remark}
  We could have used $\theta_0=0$ as we did in \Cref{alg:unconstr} which would
  have improved the final constant $R$. However, since the first bound for
  $\ss_k$ is worse this time, we would need a more complicated initial bound for
  $\ss_0$. We decided to keep it simple.
\end{remark}

\textbf{Notation.} For brevity, we write $\ss_k\opt 1$ to denote that in the
$k$-th iteration $\ss_k$ satisfies the first bound, that is
$\ss_{k} = \sqrt{\frac 2 3 + \theta_{k-1}}$. Similarly, for $\ss_k\opt 2$. Also
let $L$ be the Lipschitz constant of $\nabla f$ over the set $B(x^*, R)$. This
means that $L_k\leq L$ for all $k$.

Next few statements are not very important for the first reading, as they
only concern with a lower bound of $\sum_{i=1}^k\ss_i$. The main statement in
\Cref{th:main} is valid independently of them, so the reader can go directly there.

 \begin{lemma}\label{lem:2}
   If $\ss_k\opt 2$, then 
  $\ss_{k}\geq \frac{1}{\sqrt 2 L}$ and $\ss_{k-1}+\ss_k\geq \frac{2}{L}$.
 \end{lemma}
 \begin{proof}
  Note that in this case
  $\ss_k = \frac{\ss_{k-1}}{\sqrt{2\ss_{k-1}^2L_k^2 - 1}}$, and hence 
   $\frac{1}{\ss_{k-1}^2} + \frac{1}{\ss_k^2} = 2L_k^2$. This implies that
   $\ss_k\geq \frac{1}{\sqrt 2 L_k}\geq \frac{1}{\sqrt 2 L}$.  By AM-GM
   inequality, 
   \[\left(\frac{1}{\ss_{k-1}^2}+\frac{1}{\ss_k^2}\right)(\ss_{k-1}+\ss_k)^2\geq \frac{2}{\ss_{k-1}\ss_k}\cdot 4\ss_{k-1}\ss_k=8\]
   and the conclusion
   $\ss_{k-1}+\ss_k\geq \sqrt{\frac{8}{2L^2_k}} = \frac{2}{L_k}$ follows.
 \end{proof}

\begin{lemma}\label{lem:small_bound}
  If $\ss_0$ satisfies \eqref{eq:ss_0}, then $\ss_k\geq \frac{1}{\sqrt 3 L}$ for all $k\geq 1$.
\end{lemma}
\begin{proof}
  We use induction. For $k=1$, we have either $\ss_1 = \sqrt{\frac 23 +
    \th_0}\ss_0\geq \frac{1}{\sqrt 2 L}$ or $\ss_1\opt 2$, which in view of \Cref{lem:2} also implies $\ss_1\geq \frac{1}{\sqrt 2 L}$.
  
  Suppose that $\ss_{k-1}\geq \frac{1}{\sqrt 3 L}$  and
  we must show that $\ss_{k}\geq \frac{1}{\sqrt 3 L}$. If $\ss_k\opt 2$, then we
  are done. Therefore, suppose that $\ss_k\opt 1$. Consider two options for
  $\ss_{k-1}$. If $\ss_{k-1}\opt 1$, then $\theta_{k-1}\geq \sqrt{2/3}$. Thus,
  for $\ss_k$ we have that
  \[\ss_k \geq \sqrt{\frac 23 + \sqrt{\frac 23}}\ss_{k-1} \geq \ss_{k-1}\geq \frac{1}{\sqrt 3 L}.\]
  If $\ss_{k-1}\opt 2$, then $\ss_{k-1}\geq   \frac{1}{\sqrt 2 L}$ and hence
  \[\ss_{k} = \sqrt{\frac 23 + \theta_{k-1}}\ss_{k-1}\geq  \sqrt{\frac 23} \cdot \frac{1}{\sqrt 2 L} = \frac{1}{\sqrt 3 L},\]
  which completes the proof.
\end{proof}
\begin{remark}\label{rem:boundedness}
  It is clear from above proof that condition $\ss_0\geq \frac{1}{\sqrt 2 L_1}$
  from \eqref{eq:ss_0} was used only to give us the basis for induction. Without
  that condition, one can still show in the same way that   $\ss_k\geq \min\{\ss_0, \frac{1}{\sqrt 3 L}\}$.
\end{remark}
Summing this result from $1$ to $k$ yields $\sum_{i=1}^k\ss_i\geq \frac{k}{\sqrt 3 L}$. The stepsize in the previous section 
is lower bounded by a $\sqrt{1}{\sqrt{2}L}$, so it is natural to wonder: why is the current section contains a ``larger stepsize''?  The answer is that while we cannot
show that each individual step is larger, we still show in the next theorem that its \emph{total} length
will be lower bounded by the same quantity.

\begin{theorem}\label{th:main}
Let $f$ be convex with a locally Lipschitz gradient $\nabla f$. Then the sequence
$(x^k)$ generated by \Cref{alg:unconstr-2} converges to a solution and
\begin{equation}\label{steps_denom}
  \min_{i\in [k]} \bigl(f(x^i)-f_*\bigr) \leq    \frac{R^2}{2\sum_{i=1}^k\ss_i},
\end{equation}
where $R$ is defined as in \eqref{eq:R}.  In particular, if $\ss_0$ satisfies \eqref{eq:ss_0}, then
\begin{equation}\label{bounds_denom}
  \min_{i\in [k]} \bigl(f(x^i)-f_*\bigr) \leq  \frac{LR^2}{\sqrt 2 k},
\end{equation}
where $L$ is the Lipschitz constant of $\nabla f$ over $B(x^*, R)$.
\end{theorem}
Of course, the important bound here is \eqref{steps_denom}. The second bound only
shows that our choice of stepsizes $\ss_k$ cannot be too bad.
The bound in \eqref{bounds_denom} is stronger than the bound
$\frac{\sqrt 3LR^2}{2 k}$, which could be obtained as a direct consequence
of~\Cref{lem:small_bound}. The derivation of the sharper bound as in
\eqref{bounds_denom} is presented in \Cref{sec:better}.
\begin{proof}
We proceed in the same way as in \Cref{lem2:bounded}, but this time we keep all
the terms that were discarded earlier. Specifically,
summing~\eqref{lemma_unconstr2_energy} over all $k$ yields
\begin{align}
    \label{eq:dfgf}
      &\n{x^{k+1}-x^*}^2+ \n{x^{k+1}-x^k}^2 \notag \\
+\,     & \ss_k(2+3\th_k)(f(x^{k})-f_{*}) +\sum_{i=1}^{k-1}(\ss_i(2+3\theta_i) - 3\ss_{i+1}\theta_{i+1})(f(x^i)-f_*) \notag   \\
  \leq \, &\n{x^1-x^*}^2  + \n{x^{1}-x^{0}}^2 + 3\ss_{1}\th_{1}(f(x^{0})-f_*) \notag \\
     \leq \, & \n{x^0-x^*}^2 + 2\ss_0^2\n{\nabla f(x^0)}^2 + \ss_{0}(f(x^{0})-f_*) =  R^2,
\end{align}
where the last two bounds follow from the same arguments as
in~\Cref{lem2:bounded}. Note that each factor
$(\ss_k(2+3\theta_k) - 3\ss_{k+1}\theta_{k+1})$ is nonnegative and their sum is
\[ \ss_k(2+3\th_k) +\sum_{i=1}^{k-1}(\ss_i(2+3\theta_i)  -3\ss_{i+1}\theta_{i+1}) = 2\sum_{i=1}^k\ss_i + 3\theta_1\ss_1 \geq 2\sum_{i=1}^k\ss_i.\]
Hence, we readily obtain that
\[\min_{i\in[k]}(f(x^i)-f_*) \leq \frac{R^2}{2\sum_{i=1}^k\ss_i}.\]
In particular, if $\ss_0$ satisfies \eqref{eq:ss_0}, then inequality
\eqref{bounds_denom} is a direct consequence of \Cref{lem:improved_sum}, which
we prove in the next section.

It remains to prove that $(x^k)$ converges to a solution. Next arguments will be
similar to the ones in~\cite{malitsky2019_descent}. We have already proved that
$(x^k)$ is bounded. As $f$ is $L$-smooth over $B(x^*, R)$, we have
\[f(x^*) - f(x^k)\geq \lr{\nabla f(x^k), x^*-x^k} + \frac{1}{2L}\n{\nabla f(x^k)}^2.\]
Using this sharper bound instead of plain convexity in~\eqref{eq:adgd_main1} and
repeating the same arguments as in \Cref{lem:unconstr2_energy}, we end up with the same inequality plus the extra term 
\begin{align}\label{unconstr2_energy2}
  &\n{x^{k+1}-x^*}^2+ \n{x^{k+1}-x^k}^2 + \ss_k(2+3\th_k)(f(x^{k})-f_{*}) +\frac{\ss_k}{L}\n{\nabla f(x^k)}^2 \notag   \\
  \leq \, &\n{x^k-x^*}^2  + \n{x^{k}-x^{k-1}}^2 + 3\ss_{k}\th_{k}(f(x^{k-1})-f_*).
\end{align}
Now, by telescoping this inequality we infer that
$\sum_{i=1}^k\frac{\ss_i}{L}\n{\nabla f(x^i)}^2\leq R^2$. Since the sequence
$(\ss_k)$ is separated from $0$ (note that this is independent of condition
\eqref{eq:ss_0} by \Cref{rem:boundedness}), we conclude that
$\nabla f(x^k)\to 0$ as $k\to \infty$. Hence, all limit points of $(x^k)$ are
solutions. Applying $3\theta_k\ss_k\leq (2+3\theta_{k-1})\ss_{k-1}$ in~\eqref{unconstr2_energy2} we get
\[\n{x^{k+1}-x^*}^2 + b_{k+1}\leq \n{x^k-x^*}^2 + b_k,\]
where  $b_k = \n{x^k-x^{k-1}}^2 + \ss_{k-1}(2+3\theta_{k-1})(f(x^{k-1})-f_*) $. Then the convergence
of $(x^k)$ to a solution follows from the standard Opial-type arguments. 
\end{proof}

\subsection{Better bounds for the sum of stepsizes} \label{sec:better}
In this section we prove the bound $\sum_{i=1}^k\ss_i \geq \frac{k}{\sqrt 2 L}$. 
\begin{lemma}\label{lem:small_step} 
 If $\th_k<\frac 1 3$, then $\ss_k\mapsto 2$ and  $\ss_{k-1}L_{k} > \sqrt{5}$,
 $\ss_{k-2}L_k\geq \frac 3 2$, $\ss_{k-3}L_k\geq 1$.
\end{lemma}
\begin{proof}
  By definition, $\ss_{k}\opt 1$ means that
  $\ss_{k}=\sqrt{\frac 2 3 + \theta_{k-1}}\ss_{k-1}$ and thus
  $\theta_{k}\geq \sqrt{\frac 2 3}$. Hence, $\ss_{k}\opt 2$. Then we have that
  $\theta_{k} = \frac{1}{\sqrt{2\ss_{k-1}^2L_{k}^2 -1}} < \frac 1 3$
  which implies $\ss_{k-1}L_{k} > \sqrt 5$. Since we get a large $\ss_{k-1}$,
  the first bound on stepsizes does not allow  previous steps to be much
  smaller. That is the idea we shall use.
  
  For any $k$, we have that $\th_k\leq \sqrt{\frac 23 + \th_{k-1}} $. As
  $\th_0 \leq 1$, it is trivial to prove that
  $\th_k\leq \frac{1+\sqrt{\frac{11}{3}}}{2} \eqqcolon t_0$, which is the root of
  $t-\sqrt{\frac 23 +t} = 0$. 
  From $\ss_{k-1}L_{k} > \sqrt{5}$, it follows that
  \[ \sqrt{5}<\ss_{k-1}L_k\leq \sqrt{\frac 23 + \th_{k-2}}\ss_{k-2} L_{k}\leq t_0 \ss_{k-2}L_k.\]
  Hence, to prove $\ss_{k-2}L_k\geq \frac 32$, it only remains to check that
  $\frac{\sqrt 5}{t_0}\geq \frac 3 2$.

  Similarly, we have
  \[ \frac 32\leq \ss_{k-2}L_k\leq \sqrt{\frac 23 + \th_{k-3}}\ss_{k-3} L_{k}\leq t_0 \ss_{k-3}L_k.\]
And to prove $\ss_{k-3}L_k\geq 1$, we must check that $\frac{3}{2t_0}\geq 1$.
\end{proof}
Given the sequence $(\alpha_k)_{k\ge 1}$, we call its element $\alpha_m$ a \emph{breakpoint}, if $\theta_m<\frac{1}{3}$ and $\alpha_m < \frac{1}{L}$ . The next lemma says that a small step can only occur shortly after a breakpoint.
\begin{lemma}\label{lem:breakpointisclose}
  If $\ss_k <\frac{1}{\sqrt 2L}$, then exactly one of the following holds
  \begin{enumerate}[(i)]
  \item $\ss_{k-1}$ is a breakpoint;
  \item $\ss_{k-1}<\ss_k$ and $\ss_{k-2}$ is a breakpoint.
  \end{enumerate}

\end{lemma}
\begin{proof}
In view of \Cref{lem:2},  the statement implies that $\ss_{k}\mapsto 1$.   Suppose that $\ss_{k-1}$ is
  not a breakpoint, since otherwise we are done. This means that
  either (a) $\ss_{k-1}\geq \frac 1 L$ or (b) $\ss_{k-1}<\frac 1 L$
  and $\th_{k-1}\geq\frac 13$. In the first case we immediately get a contradiction,
  since
  $\ss_{k} = \sqrt{\frac 2 3+\th_{k-1}}\ss_{k-1}\geq \sqrt\frac{2}{3} \frac{1}{L}> \frac{1}{\sqrt 2 L}$.
  Then if we consider (b), the bound $\th_{k-1}\geq \frac 1 3$ implies that
  $\ss_{k-1}\leq \ss_{k}< \frac{1}{\sqrt 2L}$. Then we can apply the same
  arguments as above, but to $\ss_{k-1}$. This means that either $\ss_{k-2}$
  will be a
  breakpoint or we will have a chain
  $\ss_{k-2}\leq \ss_{k-1}\leq \ss_{k}<\frac{1}{\sqrt 2L}$. However, the latter
  option cannot occur, because using $\th_{k-1}\geq 1$ and
  $\ss_{k-1}\geq \frac{1}{\sqrt 3 L}$ ensure us that
  \[\ss_{k}=\sqrt{\frac 23 +\th_{k-1}}\ss_{k-1}\geq \sqrt{\frac23 + 1}\frac{1}{\sqrt 3L} = \frac{\sqrt{ 5}}{ 3L}> \frac{1}{\sqrt 2 L}.\]
\end{proof}

Although a breakpoint indicates that we are in the region with a small stepsize,
\Cref{lem:small_step} guarantees that previous steps were quite large. The next
lemma shows that in total we make a significant progress. 
\begin{lemma}\label{lem:sum5}
  If $\ss_m$ is a breakpoint, then $\sum_{j=-2}^2\ss_{m+j}>\frac{5}{L} $. 
\end{lemma}
\begin{proof}
  If $\ss_{m}$ is a breakpoint, then on one hand \Cref{lem:small_step} implies
  that $\ss_{m-1}\geq \frac{\sqrt 5}{L_m}$, $\ss_{m-2}\geq \frac{3}{2L_m}$. On
  the other hand, we have that $\ss_{m}\geq \frac{1}{\sqrt 2L}$,
  $\ss_{m+1}\geq \frac{1}{\sqrt 3L}$, and
  $\ss_{m+2}\geq \frac{1}{\sqrt 3L}$. Combining, we get
  \[\sum_{j=-2}^2\ss_{m+j}L \geq  \frac 32 +\sqrt 5 + \frac{1}{\sqrt 2} + \frac{2}{\sqrt 3} > 5.59. \]
\end{proof}

\begin{lemma}\label{lem:improved_sum} 
  If $\ss_0$  satisfies \eqref{eq:ss_0}, then
  for any $k\geq 1$ we have
  \begin{equation}
    \label{eq:2lkr}
  \sum_{i=1}^k\ss_i \geq \frac{k}{\sqrt 2 L}.
  \end{equation}
\end{lemma}
\begin{proof}
Let $\mathcal{M} = \{m \text{ is a breakpoint:
} \ss_{m+1}<\frac{1}{\sqrt 2L} \}$. We can split $\sum_{i=1}^k\ss_{i}$ into two
terms as
\begin{equation}
  \label{eq:asdasf2}
\sum_{i=1}^k \ss_i = \sum_{m\in \mathcal{M}}\sum_{j=-2}^2\ss_{m+j} + \text{rest}.  
\end{equation}
We claim that elements in the ``rest'' are greater or equal than
$\frac{1}{\sqrt 2 L}$. Indeed, if $\ss_i<\frac{1}{\sqrt 2L}$ is in the ``rest''
term, then either $\ss_{i-1}$ is a breakpoint or $\ss_{i-1}<\frac{1}{\sqrt 2L}$
and $\ss_{i-2}$ is a breakpoint, as \Cref{lem:breakpointisclose} suggests. In
either case, $\ss_i$ must be included in the
first sum, by the definition of $\mathcal{M}$.

Now let us estimate both terms. The first sum in~\eqref{eq:asdasf2} is greater than
$\frac{5|\mathcal{M}|}{L} > \frac{5|\mathcal{M}|}{\sqrt 2 L}$, by
\Cref{lem:sum5}. The total sum in the ``rest'' term is not less  than $\frac{k-5|\mathcal{M}|}{\sqrt 2 L}$. Hence, the desired
inequality follows. It has to be only noted that if $k-1 \in \mathcal{M}$, we
have to additionally consider the sum
$\sum_{j=-2}^1 \ss_{k-1+j}\geq \frac{4}{\sqrt 2 L}$, for which the bound follows
from the same arguments as in \Cref{lem:sum5}.
\end{proof}
\begin{remark}
  It is obvious that our analysis was not optimal. For instance, whenever
  $\ss_k\mapsto 2$, we could use $\ss_{k-1}+\ss_k\geq \frac{2}{L}$ instead of
  more conservative $\frac{2}{\sqrt 2 L}$. Similarly, we got a much better bound for
  every breakpoint. However, we did not want to overcomplicate an already
  tedious  examination. We leave it as an open question if one can provide a
  bound closer to $ \frac{k}{L}$ (or better?) with a readable proof.
\end{remark}

\section{Adaptive proximal gradient method}\label{sec:prox}
In this section we turn to a more general problem of  a composite optimization problem
\begin{equation}
  \label{eq:comp}
  \min_{x}F(x)\coloneqq f(x) + g(x),
\end{equation}
where $g\colon \R^d\to \eR$ is a convex lsc function and $f\colon \R^d \to \R$
is a convex differentiable function with locally Lipschitz
$\nabla f$. Additionally, we assume that $g$ is prox-friendly, that is we can
efficiently compute its proximal mapping $\prox_g = (\id + \partial g)^{-1}$. As before, we suppose that \eqref{eq:comp} has a solution.
 \begin{algorithm}
 \caption{Adaptive proximal gradient method}
 \label{alg:prox}
 \begin{algorithmic}[1]
   \STATE \textbf{Input:} $x^0 \in \dom g$, $\th_{0}=\frac 1 3$, $\ss_0>0$
   \STATE $x^1 =\prox_{\ss_0 g}(x^0 -\ss_0 \nabla f(x^0))$ 
   \FOR{$k = 1,2,\dots$}
   \STATE $L_k = \dfrac{\n{\nabla f(x^{k})-\nabla f(x^{k-1})}}{\n{x^{k}-x^{k-1}}}$
    \STATE
    $\ss_k = \min\left\{ \sqrt{\frac{2}{3} + \theta_{k-1}}\ss_{k-1}, \frac{\ss_{k-1}}{\sqrt{[2\ss_{k-1}^2L_k^2-1]_{+}}}\right\}$
    \hfill  {\color{gray} // $[t]_+\coloneqq\max(t, 0)$}
\STATE $x^{k+1} = \prox_{\ss_k g}(x^k - \ss_k \nabla f(x^k))$
\STATE $\theta_k = \frac{\ss_k}{\ss_{k-1}}$
    \ENDFOR
 \end{algorithmic}
\end{algorithm}

As one can notice, the only difference between \Cref{alg:prox} and
\Cref{alg:unconstr-2} is the presence of the proximal mapping. The analysis, however, is not
a straightforward generalization.

Recall that the second bound for the stepsize $\ss_k$ is equivalent to
\begin{equation}
  \label{eq:step_prox}
  \ss_k^2\left(L_k^2 - \frac{1}{2\ss_{k-1}^2}\right)\leq \frac 1 2.
\end{equation}
We can rewrite $x^{k+1} = \prox_{\ss_k g}(x^k - \ss_k \nabla f(x^k))$ as an implicit equation
\begin{equation}
  \label{eq:implicit}
  x^{k+1} = x^k - \ss_k (\nabla f(x^k) + \sbg{k+1}),
\end{equation}
where $\sbg{k+1}$ is a certain subgradient of $g$ at $x^{k+1}$, that is $\sbg{k+1} \in \partial g(x^{k+1})$. For this particular subgradient we will also use notation
\[\Sbg{k} = \nabla F(x^k)+\sbg{k}.\]
First, we adapt our basic inequality~\eqref{eq:adgd_main1} to the more general case. By prox-inequality, we have
\begin{align}\label{eq:1234d}
  \lr{x^{k+1}-x^k + \ss_k \nabla f(x^k), x-x^{k+1}} \geq \ss_k(g(x^{k+1})-g(x)), \quad \forall x.
\end{align}
Then we set $x=x^*$ above and  transform it into
\begin{align}\label{eq:ldkjfpo}
  \n{x^{k+1}-x^*}^2+2\ss_k (g(x^{k+1})-g(x^*)) \leq \n{x^k-x^*}^2 + 2\ss_k \lr{\nabla f(x^k), x^*-x^{k+1}} - \n{x^{k+1}-x^k}^2.
\end{align}
This standard inequality is at the heart of the analysis of the proximal
gradient method. To complete the full proof, or rather to get the final
inequality, the classical analysis only requires applying one convexity inequality and
one descent lemma to function $f$. Our analysis, however, will be different.  The main
nuisance is that in the $k$-th iteration the proximal map yields us $g(x^{k+1})-g(x^*)$ term, while our
adaptivity approach works with $f(x^k)-f(x^*)$, as we remember from before. Thus,
our first obstacle is to understand how to combine these two terms.

First, we estimate the term $\lr{\nabla f(x^k), x^*-x^{k+1}}$ in the RHS of~\eqref{eq:ldkjfpo}.
We have
\begin{align*}
   \lr{\nabla f(x^k), x^*-x^{k+1}}  &= \lr{\nabla f(x^k), x^*-x^{k}} + \lr{\nabla f(x^k), x^k-x^{k+1}} \\
  &=\lr{\nabla f(x^k), x^*-x^{k}}+ \lr{\nabla f(x^k) + \sbg{k}, x^k-x^{k+1}} + \lr{\sbg{k}, x^{k+1}-x^k} \\
  &\leq f(x^*)-f(x^k) +\lr{\nabla f(x^k) + \sbg{k}, x^k-x^{k+1}} + g(x^{k+1})-g(x^k),
\end{align*}
where in the last inequality we used separately convexity of $f$ and
$g$. Applying this inequality in~\eqref{eq:ldkjfpo} yields
\begin{align}\label{eq:basic_prox}
  & \,   \n{x^{k+1}-x^*}^2+2\ss_k (F(x^{k})-F(x^*)) \notag \\
  \leq &\, \n{x^k-x^*}^2 + 2\ss_k\lr{\nabla f(x^k) + \sbg{k}, x^k-x^{k+1}} - \n{x^{k+1}-x^k}^2 \notag\\
  \leq &\,\n{x^k-x^*}^2 + \ss_k^2\n{\nabla f(x^k) + \sbg{k}}^2.
\end{align}
As we see, the final inequality is very much in the spirit of
\eqref{eq:adgd_main1}.

\begin{lemma}[Compare to~\Cref{prop1}]\label{prop1_prox}
  For iterates $(x^{k})$ with arbitrary stepsizes, it holds
  \begin{equation}\label{eq:scalar_prod_to_norm}
    \lr{\nabla f(x^{k}) + \sbg{k}, \nabla f(x^{k-1}) + \sbg{k}} \leq \n{\nabla f(x^{k-1}) + \sbg{k}}^{2}.
  \end{equation}
\end{lemma}
\begin{proof}
  As before, this is just monotonicity of $\nabla f$ in disguise:
  \begin{align*}
    &\n{\nabla f(x^{k-1})+\sbg{k}}^{2} - \lr{\nabla f(x^{k}) +\sbg{k}, \nabla f(x^{k-1}) + \sbg{k}} \\
    &\qquad= \lr{\nabla f(x^{k-1})-\nabla f(x^{k}), \nabla f(x^{k-1}) + \sbg{k}} \\
        &\qquad =  \frac{1}{\ss_{k-1}}\lr{\nabla f(x^{k-1})-\nabla f(x^{k}), x^{k-1}-x^{k}} \geq 0.\qedhere
  \end{align*}
\end{proof}

The next lemma is special for the composite case. Although it looks like this fact should be known in the literature, we were not able to identify it.
\begin{lemma}\label{lem:subgrad_norm}
  For iterates $(x^{k})$ of the proximal gradient method with arbitrary stepsizes, it holds
        \begin{equation}\label{eq:subgrad_prox}
                \|\nabla f(x^k) + \sbg{k+1}\| \le \|\nabla f(x^k) + \sbg{k}\|.
        \end{equation}
\end{lemma}
\begin{proof}
This time it is just a monotonicity of $\partial g$ in disguise:
        \begin{align*}
                \|\nabla f(x^k) + \sbg{k}\|^2
                &= \|\nabla f(x^k) + \sbg{k+1} + \sbg{k} - \sbg{k+1}\|^2 \\
                &\overset{\eqref{eq:implicit}}{=} \left\|\frac{1}{\ss_k}(x^k-x^{k+1}) + \sbg{k} - \sbg{k+1}\right\|^2 \\
                &= \frac{1}{\ss_k^2}\|x^k-x^{k+1}\|^2 + \frac{2}{\ss_k}\lr{x^k - x^{k+1}, \sbg{k}-\sbg{k+1}} + \|\sbg{k}-\sbg{k+1}\|^2 \\
                &\geq \frac{1}{\ss_k^2}\|x^k-x^{k+1}\|^2 + \frac{2}{\ss_k}\lr{x^k - x^{k+1}, \sbg{k}-\sbg{k+1}}\\
                &\geq \|\nabla f(x^k) + \sbg{k+1}\|^2,
        \end{align*}
        where the last inequality follows from monotonicity of $\partial g$ and \eqref{eq:implicit}.
\end{proof}
In~\Cref{sec:larger} we estimated
$\n{x^{k+1}-x^k}^2=\ss_k^2\n{\nabla f(x^k)}^2$. This time, $\n{x^{k+1}-x^k}^2$
and $\ss_k^2\n{\Sbg{k}}^2$ are  different and it is the latter term that matters
to us.

\begin{lemma}[Compare to \Cref{lem:decr3}]\label{lem:decr3_constr}
  For  iterates $(x^k)$ of \Cref{alg:prox} it holds
   \begin{align*}
  \ss_k^2\n{\Sbg{k}}^{2} \leq \frac{\ss_{k-1}^2}{2} \n{\Sbg{k-1}}^{2} + \frac 32 \ss_{k}\th_{k}(F(x^{k-1})-F(x^k)).
\end{align*}
 \end{lemma}
 \begin{proof}
The main idea of the proof is exactly the same as
in~\Cref{lem:decr3}. However, the presence of $\sbg{k}$ make it slightly
more cumbersome. The previous two lemmata are instrumental on our way. We have
\begin{align*}
  &\quad \ss_{k}^{2}\|\nabla f(x^{k}) +\sbg{k}\|^{2}
  = \ss_{k}^{2}\n{\nabla f(x^{k})-\nabla f(x^{k-1})}^{2} -\ss_{k}^{2}\n{\nabla f(x^{k-1})+\sbg{k}}^{2} \notag\\
  &\qquad\qquad + 2\ss_{k}^{2}\lr{\nabla f(x^{k})+\sbg{k}, \nabla f(x^{k-1})+\sbg{k}}\notag \\
  &= \ss_k^2 L_k^{2} \n{x^{k}-x^{k-1}}^{2}  -\frac{\ss_{k}^{2}}{2\ss_{k-1}^2}\n{x^k-x^{k-1}}^{2}
    -\frac{\ss_{k}^{2}}{2}\n{\nabla f(x^{k-1})+\sbg{k}}^{2}\notag \\
  &\qquad\qquad + 2\ss_{k}^{2}\lr{\nabla f(x^{k})+\sbg{k}, \nabla f(x^{k-1})+\sbg{k}} \notag\\
  &\overset{\eqref{eq:scalar_prod_to_norm}}{\leq}\ss_k^2\left(L_k^2 -\frac{1}{2\ss_{k-1}^2}\right)
    \n{x^{k}-x^{k-1}}^{2}  + \frac {3\ss_{k}^{2}}{2}\lr{\nabla f(x^{k})+\sbg{k}, \nabla f(x^{k-1})+\sbg{k}}\notag\\
 &\overset{\eqref{eq:step_prox}\&\eqref{eq:implicit}}{\leq} \frac 12  \n{x^{k}-x^{k-1}}^{2}  + \frac {3}{2}\ss_{k}\th_{k}\lr{\nabla f(x^{k})+\sbg{k}, x^{k-1}-x^{k}}  \notag\\
 &\overset{\eqref{eq:implicit}}{=} \frac{\ss_{k-1}^2}{2}  \n{\nabla f(x^{k-1})+\sbg{k}}^2  + \frac {3}{2}\ss_{k}\th_{k}\lr{\Sbg{k}, x^{k-1}-x^{k}} \notag \\
  &\overset{\eqref{eq:subgrad_prox}}{\leq}\frac{\ss_{k-1}^2}{2}  \n{\nabla f(x^{k-1})+\sbg{k-1}}^2 +  \frac {3}{2}\ss_{k}\th_{k}\lr{\Sbg{k}, x^{k-1}-x^{k}}.
\end{align*}
Convexity of $F$ completes the proof.
\end{proof}

\begin{lemma}[Compare to \Cref{lem:unconstr2_energy}]\label{lem:constr_energy}
  For iterates $(x^k)$ of \Cref{alg:prox} and any solution $x^*$ it
  holds
\begin{align}\label{lemma_constr_energy}
  &\n{x^{k+1}-x^*}^2+ \ss_k^2\n{\Sbg{k}}^2 + \ss_k(2+3\th_k)(F(x^{k})-F_{*})\notag   \\
  \leq \, &\n{x^k-x^*}^2  + \ss_{k-1}^2\n{\Sbg{k-1}}^2 + 3\ss_{k}\th_{k}(F(x^{k-1})-F_*).
\end{align}
\end{lemma}
\begin{proof}
 The same as in~\Cref{lem:unconstr2_energy}.
\end{proof}
We define $R$ in the same way as in \eqref{eq:R}
\begin{equation}
    \label{eq:R2}
  R^2 = \n{x^0-x^*}^2  + 2\ss_0^2\n{\Sbg{0}}^2+\ss_{0}(F(x^{0})-F_*).
  \end{equation}
  \begin{lemma}\label{lem2:bounded-2}
  The sequence $(x^k)$ is bounded. In particular, for any solution 
  $x^*$ of \eqref{eq:comp} we have  $x^k\in B(x^*, R)$.
\end{lemma}
\begin{proof}
  The same as in \Cref{lem2:bounded}. We use \eqref{lemma_constr_energy} to telescope until $k=1$
  and then apply \eqref{eq:basic_prox} with $k=0$ to bound $\n{x^1 - x^*}^2$.
\end{proof}
\begin{theorem}\label{th:main-2}
  Let $f$ be convex with a locally Lipschitz gradient $\nabla f$ and $g$ be
  convex lsc. Then the sequence $(x^k)$ generated by \Cref{alg:prox} converges
  to a solution of~\eqref{eq:comp} and
\begin{equation}\label{steps_denom-2}
  \min_{i\in [k]} \bigl(F(x^i)-F_*\bigr) \leq    \frac{R^2}{2\sum_{i=1}^k\ss_i}.
\end{equation}
In particular, if $\ss_0$ satisfies \eqref{eq:ss_0}, then
\begin{equation}\label{bounds_denom-2}
  \min_{i\in [k]} \bigl(F(x^i)-F_*\bigr) \leq  \frac{LR^2}{\sqrt 2 k},
\end{equation}
where $L$ is the Lipschitz constant of $\nabla f$ over $B(x^*, R)$.
\end{theorem}
\begin{proof}
  The proof of inequalities~\eqref{steps_denom-2} and~\eqref{bounds_denom-2} is
  almost identical to the one in \Cref{th:main}. The proof of convergence of
  $(x^k)$ to a solution is, however, more nuanced. The nontrivial part is to
  show that all limit points of $(x^k)$ are solutions. While on the surface, it
  should be no harder than before, the fact that $\lim_{k\to +\infty}\ss_k$ can
  be $+\infty$ complicates things a bit.
  
  Let $x^*$ be a solution of \eqref{eq:comp}. By $L$-smoothness of $f$ over
  $B(x^*, R)$, we have
  \begin{align*}
  f(x^*) - f(x^k)& \geq \lr{\nabla f(x^k), x^*-x^k} + \frac{1}{2L}\n{\nabla f(x^k)-\nabla f(x^*)}^2.  
  \end{align*}
Using this improved bound, similarly as it was done in \eqref{unconstr2_energy2}, we get
\begin{align}\label{constr2_energy2}
  &\n{x^{k+1}-x^*}^2+ \ss_k^2\n{\Sbg{k}}^2 + \ss_k(2+3\th_k)(F(x^{k})-F_{*}) +\frac{\ss_k}{L}\n{\nabla f(x^k)-\nabla f(x^*)}^2 \notag   \\
  \leq \, &\n{x^k-x^*}^2  + \ss_{k-1}^2\n{\Sbg{k-1}}^2 + 3\ss_{k}\th_{k}(F(x^{k-1})-F_*).
\end{align}
By telescoping this inequality as before, we can now additionally infer that
\begin{equation}
  \label{eq:sd1}
\sum_{k=1}^\infty \ss_k \n{\nabla f(x^k)-\nabla f(x^*)}^2<+\infty  
\end{equation}
and thus, $\n{\nabla f(x^k) - \nabla f(x^*)}\to 0$. Specifically, this implies $\nabla f(x^k) - \nabla f(x^{k-1})\to 0$ as $k\to \infty$.

We want to prove that all limit points of $(x^k)$ are solutions. To this end, 
we will use prox-inequality \eqref{eq:1234d} rewritten as 
\begin{align}\label{eq:0we9}
  \frac{1}{\ss_k}\lr{x^{k+1}-x^k,  x-x^{k+1}} + \lr{\nabla f(x^{k}), x-x^{k+1}}
  \geq g(x^{k+1})-g(x), \forall x
\end{align}
which in turn, by convexity of $f$, leads to
\begin{align}
  \label{eq:2ftsd}
    \frac{1}{\ss_k}\lr{x^{k+1}-x^k , x-x^{k+1}} + \lr{\nabla f(x^k)-\nabla f(x^{k+1}),x-x^{k+1}} \geq F(x^{k+1}) - F(x).
\end{align}
The left-hand side has two terms, and the second term evidently tends to $0$ as
$\nabla f(x^{k+1})-\nabla f(x^k)\to 0$.  If we can show the same for the first term, it will imply that all limit points of $(x^k)$ are solutions.

Consider~\eqref{eq:0we9} again, but this time we set $x=x^k$. This yields
\begin{align*}
  -\frac{1}{\ss_k}\n{x^{k+1}-x^k}^2 + \lr{\nabla f(x^{k}), x^k-x^{k+1}}
  \geq g(x^{k+1})-g(x^k).
\end{align*}
We manipulate the  inequality above as follows
\begin{align*}
  \frac{1}{\ss_k}\n{x^{k+1}-x^k}^2
  &\leq  \lr{\nabla f(x^{k}), x^k-x^{k+1}}+ g(x^{k})-g(x^{k+1}) \\
  &= \lr{\nabla f(x^{k})-\nabla f(x^*), x^k-x^{k+1}}+ \underbrace{\lr{\nabla
    f(x^*),x^k-x^{k+1}} +  g(x^{k})-g(x^{k+1})}_{\delta_k}\\
  &\leq \frac{\ss_k}{2}\n{\nabla f(x^{k})-\nabla f(x^*)}^2 +
    \frac{1}{2\ss_k}\n{x^{k+1}-x^k}^2 + \delta_k,
\end{align*}
where in the last inequality we applied Cauchy-Schwarz and Young's inequalities.
From this we deduce that
\[\frac{1}{\ss_k}\n{x^{k+1}-x^k}^2 \leq \ss_k \n{\nabla f(x^{k})-\nabla
    f(x^*)}^2 + 2\delta_k.\]
Note that the sequence $\lrb{\ss_k \n{\nabla f(x^{k})-\nabla
    f(x^*)}^2}$ is summable by \eqref{eq:sd1}.  Also, the sequence $(\delta_k)$ is
summable, since $(x^k)$ is bounded and $g(x^k)$ is lower-bounded: $g(x^k)\geq
F_* - f(x^k)>-\infty$ for all $k$. Hence,
$\sum_{k}\frac{1}{\ss_k}\n{x^{k+1}-x^k}^2<+\infty$ and, thus,
$\frac{1}{\ss_k}\n{x^{k+1}-x^k}^2\to 0$ as $k\to +\infty$. Given that $(\a_k)$ is separated from zero, it immediately follows that $\frac{1}{\ss_k}\n{x^{k+1}-x^k}\to 0$ as well.

Therefore, we have proved that all limit points of $(x^k)$ are solutions.  The
proof of convergence of the whole sequence $(x^k)$ runs as before in
\Cref{th:main}. 
\end{proof}
\begin{remark}
We didn't derive a linear convergence of the adaptive proximal gradient,
  when $F$ is strongly convex. We only mention that it is quite straightforward
  and goes along the same lines as the original AdGD
  in~\cite[Theorem 2]{malitsky2019_descent} in the strongly convex regime.
\end{remark}

\section{Literature and discussion}\label{sec:literature}

\paragraph{Linesearch.} There are many variants of linesearch procedures that go
back to celebrated works of Goldstein~\cite{goldstein1962cauchy} and
Armijo~\cite{armijo1966}. We discuss an efficient implementation of the latter
in detail in the next section. For other variants of linesearch, we refer to
\cite{bello2016convergence,salzo2017variable}.

\paragraph{Adagrad-type methods.} Original Adagrad algorithm was proposed
simultaneously in~\cite{duchi2011adaptive} and \cite{mcmahan2010adaptive}. The method
has had a stunning impact on machine learning applications. It has also spawned a stream of various extensions that retain the same idea of
using eventually decreasing steps. Because of this, its adaptivity is more prominent in the
non-smooth regime, where stepsizes must be diminishing to guarantee convergence. Recent works~\cite{defazaio2023learning,ivgi2023dog} have proposed ways to increase $d_k$ in the update~\eqref{eq:adagrad} and \cite{khaled2023dowg} even proved convergence of such method on smooth objectives. However, the stepsize in these method still eventually stops increasing as $d_k$ is upper bounded by a constant.

In addition,
Adagrad-type methods are usually sensitive to the initialization, as they either degrade in performance when $d_k=D$ and $D$ is not chosen carefully, or their convergence rate depends multiplicatively on $\log(\|x^0 - x^*\|/d_0)$. In contrast, in our methods, the cost of estimating $\alpha_0$ to satisfy condition~\eqref{eq:ss_0} is additive and its impact vanishes as the total number of iterations increases.

\paragraph{Refined results on GD with a fixed stepsize.}
Paper~\cite{teboulle2022elementary} summarizes quite well the difficulty of GD
analysis with large steps. In it, the authors derive sharp convergence bounds
separately for two cases
$\ss L\in (0,1]$ and $\ss L\in (1,2)$, and the latter case is considerably
harder. In our analysis it is even harder, since the steps can go far beyond the global
upper bound $\frac{2}{L}$. A  surprising recent
result~\cite{grimmer2023provably} showcases how actually little is understood in
this case.

\paragraph{Small gradient.}
The lack-of-descent property makes it hard to deduce the $\cO(1/k)$ rate for
the last-iterate $\n{\nabla f(x^k)}$, which is known for GD with a fixed stepsize. We leave it as
an open problem to establish a rate.

\paragraph{Extensions.}
Because the analysis of the algorithm is so special, it is not easy to extend it to basic
generalizations of GD. However, some works have already build upon it. 
In~\cite{vladarean2021}, the authors consider a convex smooth minimization
subject to linear constraints and combined the adaptive GD \cite{malitsky2019_descent}
with  the Chambolle-Pock algorithm~\cite{chambolle2011first}. The authors of~\cite{latafat2023adaptive}
went even further and considered a more general composite minimization problem
subject to linear constraints, where the same two ideas as before were combined
with a novel way of handling the prox mapping.

If we consider variational inequalities settings in the monotone case, then it
is not clear how such adaptivity can help, since the most natural extension, the
\emph{forward-backward} method will diverge. On the other hand, an adaptive
golden ratio algorithm~\cite{Malitsky2019}, which is the precursor of a given
work, already has all the properties that AdProxGD has.

\section{Experiments}\label{sec:exp}
\newcommand{\mylegend}{\begin{center}
   \small{
     \block{0}{0}{0} AdProxGD\quad
     \block{0.0}{159.0}{129.0} (1.2, 0.5)\quad
     \block{255.0}{90.0}{175.0} (1.5, 0.8)\quad
     \block{0.0}{252.0}{207.0} (1.1, 0.5)\quad
     \block{132.0}{0.0}{205.0} (1.2, 0.9)\\
     \block{0.0}{141.0}{249.0} (1.1, 0.9)\quad
     \block{0.0}{194.0}{249.0}  (1.5, 0.5)\quad
     \block{255.0}{178.0}{253.0} (1.2, 0.8)\quad
     \block{164.0}{1.0}{34.0} (1.1, 0.8)\quad
     \block{226.0}{1.0}{52.0} (1.5, 0.9)
   }
 \end{center}
}
In the experiments\footnote{\href{https://github.com/ymalitsky/AdProxGD}{https://github.com/ymalitsky/AdProxGD}} we compare our method to the ProxGD with Armijo's linesearch.
We believe it is the best and arguably the most popular alternative to our method. 
An efficient implementation of Armijo's linesearch requires two parameters,
$s>1$ and $r<1$. In the $k$-th iteration, the first iteration of linesearch
starts from $\ss_k = s \ss_{k-1}$, that is, we want to try a slightly larger step
than in the previous iteration. If linesearch does not terminate, we start
decreasing a stepsize geometrically with  a
ratio $r$. Formally, we are looking for the largest $\ss_k =sr^i \ss_{k-1}$, for $i
=0,1,\dots$, such that for $x^{k+1}=\prox_{\ss_k g}(x^k - \ss_k \nabla f(x^k))$ it
holds that 
\[f(x^{k+1}) \leq f(x^k) + \lr{\nabla f(x^k), x^{k+1}-x^k} + \frac{1}{2\ss_k} \n{x^{k+1}-x^k}^2.\]
It is evident that each iteration of this linesearch requires one evaluation of
$f$ and $\prox_g$.
However, it is important to highlight that in some cases, the last evaluation  of
$f(x^{k+1})$ (during
linesearch) may not incur any additional costs, as certain expensive operations, such
as matrix-vector multiplication, can be reused to compute the next gradient $\nabla f(x^{k+1})$. 
Throughout our comparisons, we consistently took these factors into account and
reported only essential operations that cannot be further reused.

Our legend will stay the same for all plots:
\mylegend
where each pair of numbers represent $(s,r)$ for ProxGD with linesearch
described above. As we will see, the choice of $(s,r)$ matters a lot.

\paragraph{Maximum likelihood estimate of the information matrix.}
We consider \cite[Equation (7.5)]{boyd2004convex}, where our goal is to estimate
the inverse of a covariance matrix $Y$ subject to eigenvalue bounds. Formally, this problem can be formulated as follows
\begin{equation}
  \label{eq:mle}
  \min_{X\in \mathbbm{S}^{n}} f(X) = \log\det X - \tr(XY)\quad \text{subject to}\quad lI\preccurlyeq X \preccurlyeq u I,
  \end{equation}
  where $\mathbbm{S}^n$ denotes the space of $n$-by-$n$ symmetric matrices and
  $A\preccurlyeq B$ means that $B-A$ is positive semidefinite.

  Computing
  projection onto the constraint set
  $\cC = \{X\colon lI\preccurlyeq X \preccurlyeq u I\}$ requires computing matrix
  eigendecomposition. However, it is noteworthy that once the eigendecomposition
  is computed, both the objective and gradient evaluations can be carried out at
  a low cost. Consequently, when comparing methods, we only emphasized the
  number of projections conducted. We generated a random $y\in \R^{n}$ with entries from
  $N(0,10)$ and $\d_i\in \R^n$ with entries from $N(0,1)$, and then set $y_i = y+\delta_i$,
  for $i=1,\dots, M$. Then we computed
  $Y = \frac{1}{M}\sum_{i=1}^{M} y_iy_i^\top$. The results are presented in~\Cref{fig:mle}.

\begin{figure}[ht]
  \mylegend
  \centering
  \begin{subfigure}[b]{0.45\linewidth}
    \centering
    \includegraphics[width=\linewidth]{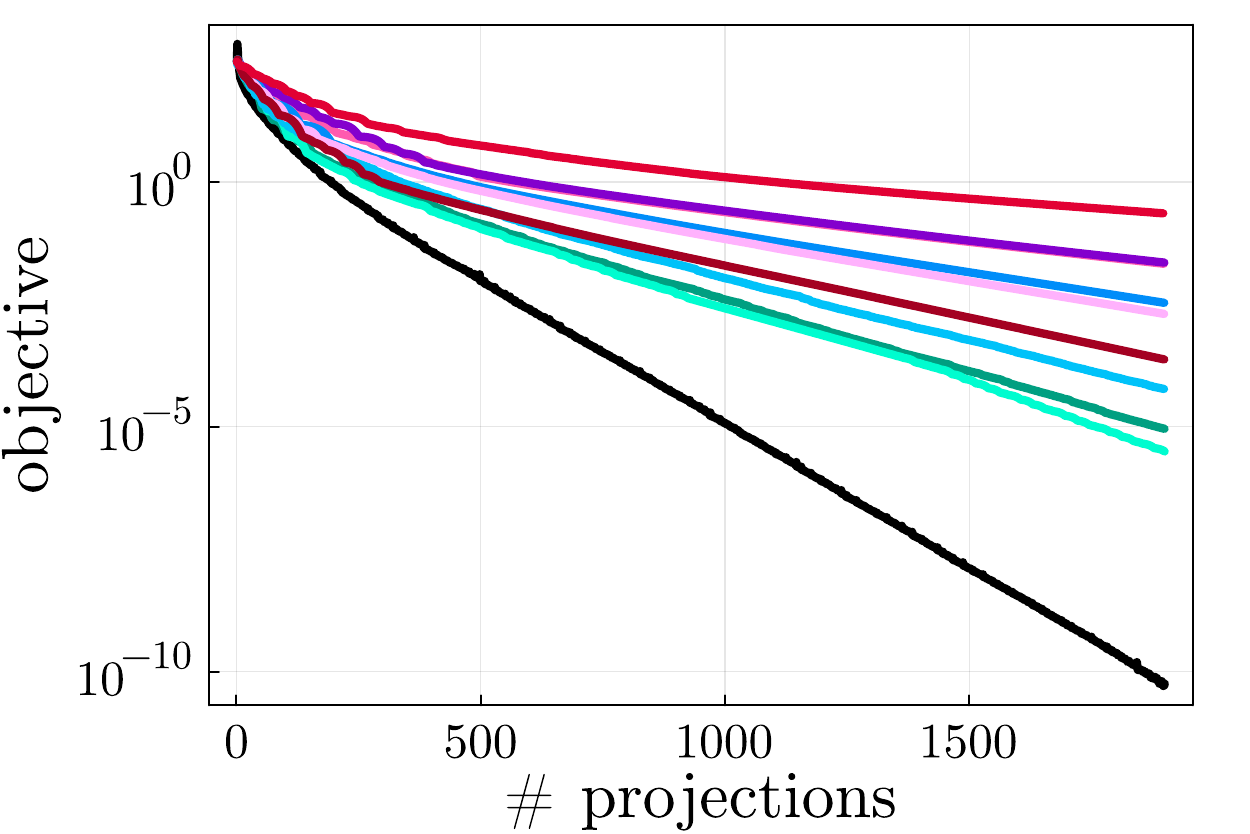}
    \caption{$n=100$, $l=0.1$, $u=10$, $M=50$}\label{fig:mle-1}
  \end{subfigure}
  \begin{subfigure}[b]{0.45\linewidth}
    \centering
    \includegraphics[width=\linewidth]{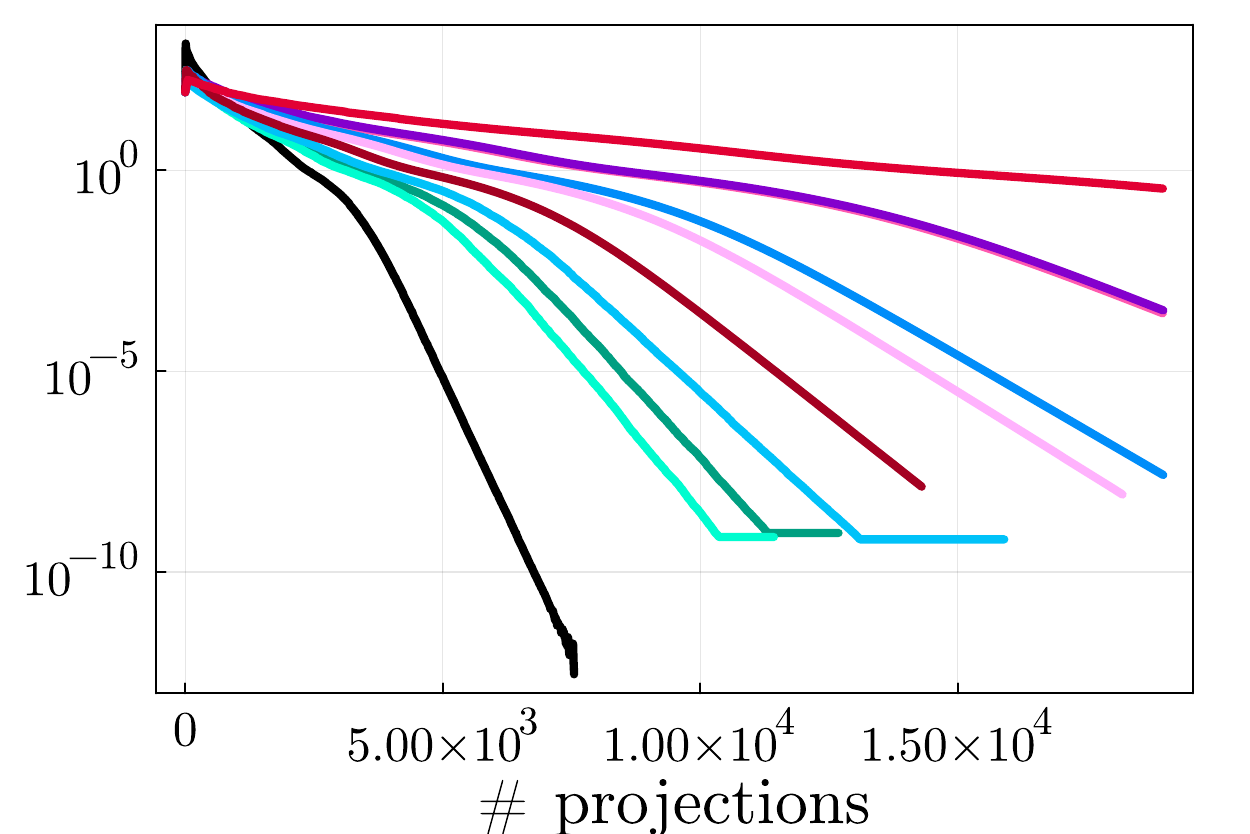}
    \caption{$n=50$, $l=0.1$, $u=1000$, $M=100$}\label{fig:mle-2}
  \end{subfigure}   
  \caption{{Maximum likelihood estimate, problem~\eqref{eq:mle}}}
  \label{fig:mle}
\end{figure}

\paragraph{Low-rank matrix completion}
We consider a famous low-rank matrix completion problem in the form 
\begin{equation}
  \label{eq:lrmc}
\min_{X\in \R^{n\times n}}\frac 12 \n{P_\Omega(X - A)}^2_{F} \quad \text{subject to}\quad
  \n{X}_{*}\leq r,
\end{equation}
where $\Omega$ is a subset of indices $(i,j)$ and $r$ is the supposed maximum rank. 
To project onto the spectral ball $\cC = \{X\colon \n{X}_{*}\leq r\}$, computing the singular value decomposition (SVD) is required, making it the most computationally expensive operation in this setting.

We created matrix $A$ by multiplying matrices $U$ and $V^\top$, where $U$ and $V$ are $n$-by-$r$ matrices with entries sampled from a normal distribution. 
The subset $\Omega$ was randomly chosen as a fraction of $\frac{1}{5}n^2$ entries from $[n]\times [n]$. The obtained results are depicted in \Cref{fig:lrmc}, where we solely compared the number of computed SVDs.

\begin{figure}[ht]
  \mylegend
  \centering
  \begin{subfigure}[b]{0.45\linewidth}
    \centering
    \includegraphics[width=\linewidth]{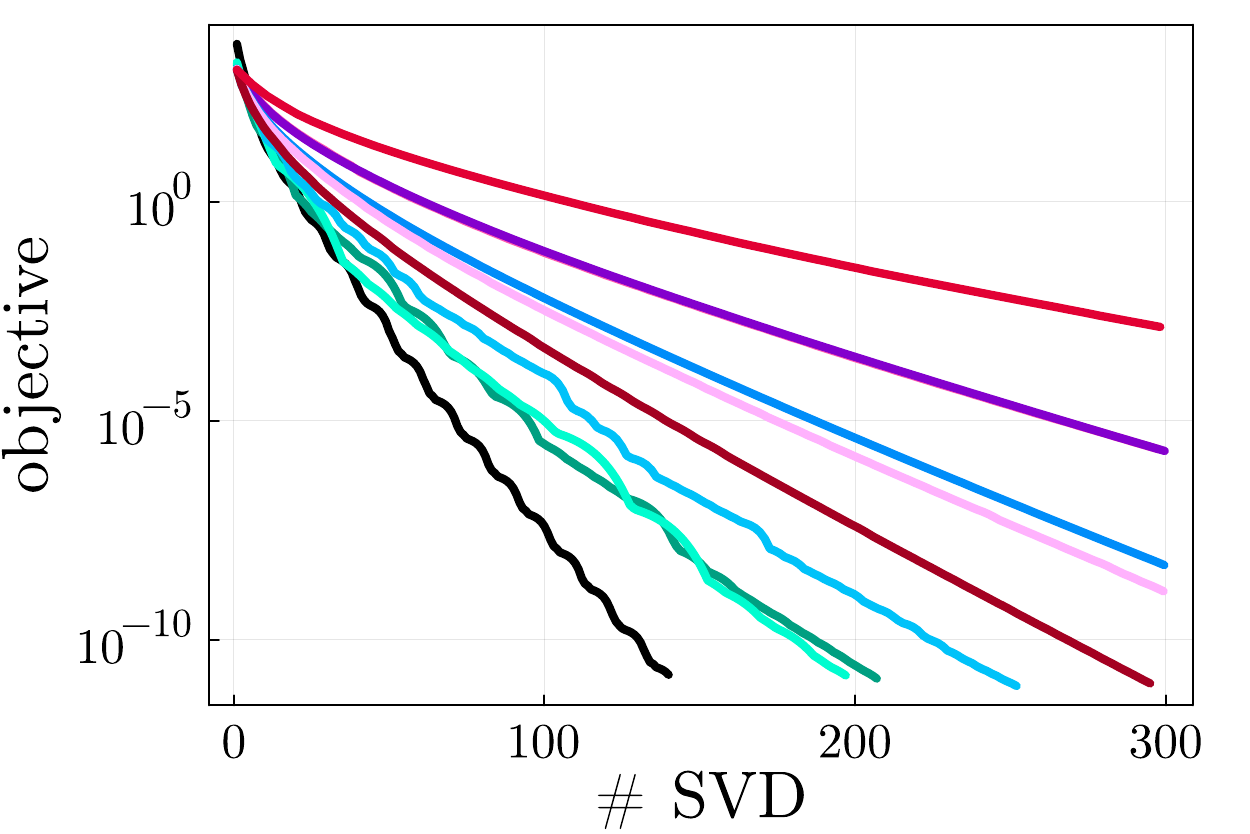}
    \caption{$n=100$, $r=20$}\label{fig:lrmc-1}
  \end{subfigure}
  \begin{subfigure}[b]{0.45\linewidth}
    \centering
    \includegraphics[width=\linewidth]{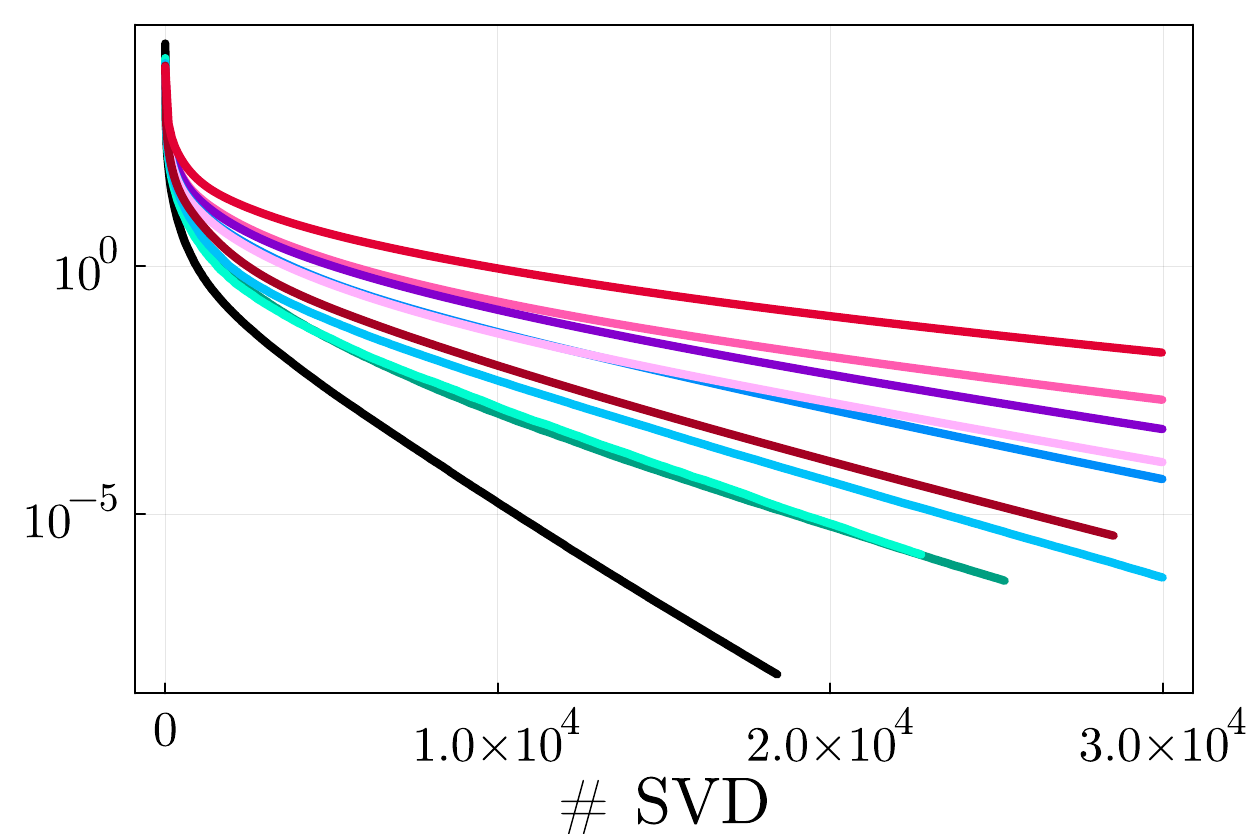}
    \caption{$n=200$, $r=20$}\label{fig:lrmc-2}
  \end{subfigure}
   
  \caption{{Low-rank matrix completion, problem~\eqref{eq:lrmc}}}
  \label{fig:lrmc}
\end{figure}

\paragraph{Minimal length piecewise-linear curve subject to linear constraints.}
We consider \cite[Example 10.4]{boyd2004convex}, where we want to minimize the length of a
piecewise-linear curve passing through $n$ points in $\mathbb{R}^2$ with
coordinates $(1,x_1), \dots, (n,x_n)$ while satisfying linear constraints
$Ax = b$, where $x=(x_1,\dotsc, x_n)$. Given $A\in \R^{m\times n}$ and $b\in \R^m$, this can be modeled as 
\begin{equation}
  \label{eq:mlc}
 \min_{x\in \R^n} (1+x_1)^{\nicefrac 12} +
  \sum_{i=1}^{n-1}(1+(x_{i+1}-x_i)^2)^{\nicefrac 1 2}\quad \text{subject to}\quad Ax =
  b.  
\end{equation}
While applying the proximal gradient method,
the most computationally expensive operation is computing the projection onto
$\cC= \{x\colon Ax=b\}$.  Assuming that $A$ is full rank with $m\leq n$, this
projection can be computed as $P_\cC z = z - A^\top (AA^\top)^{-1}(Az - b)$.

In comparison, we focused solely on the number of computed projections. We
generated a random $m$-by-$n$ matrix $A$ and random vector $w$ with entries sampled from a normal distribution and set $b = Aw$.

\begin{figure}[ht]
  \mylegend
  \centering
  \begin{subfigure}[b]{0.45\linewidth}
    \centering
    \includegraphics[width=\linewidth]{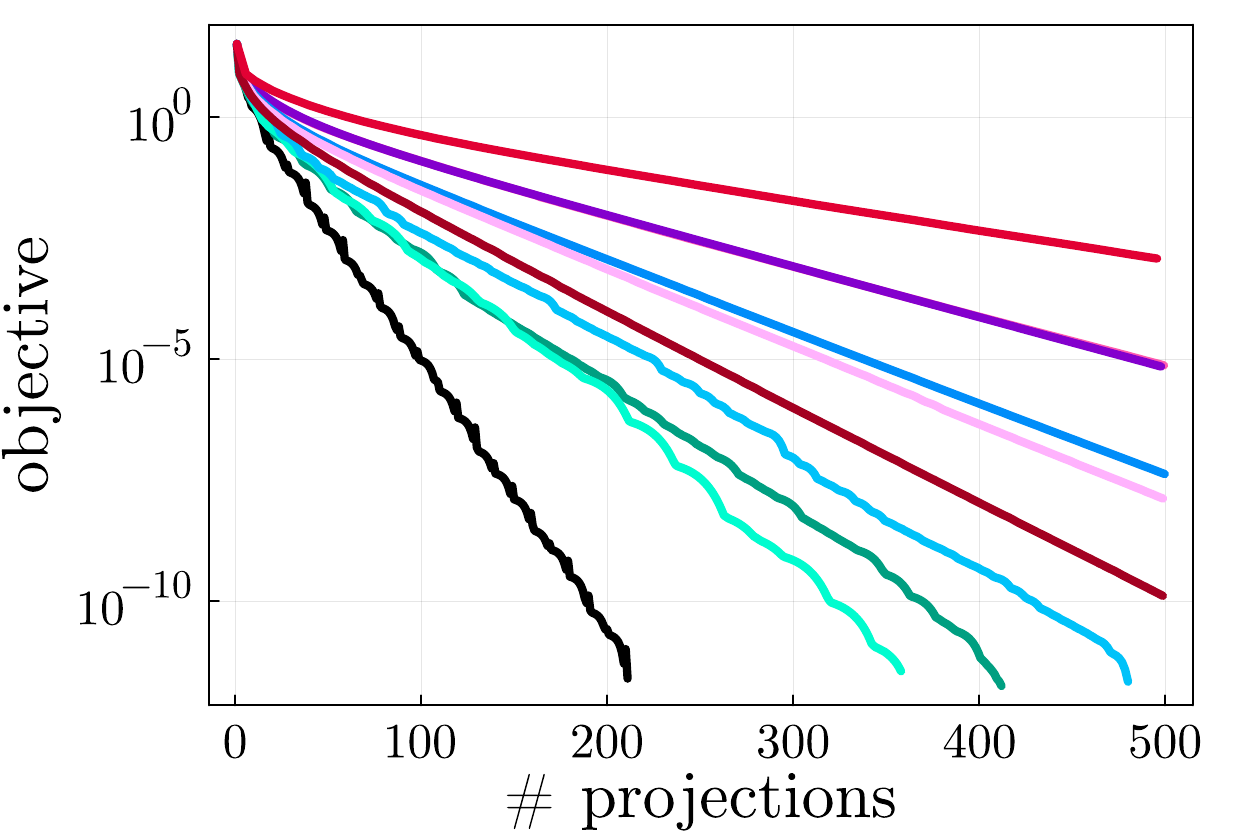}
    \caption{$m=50$, $n=200$}\label{fig:mlc-1}
  \end{subfigure}
  \begin{subfigure}[b]{0.45\linewidth}
    \centering
    \includegraphics[width=\linewidth]{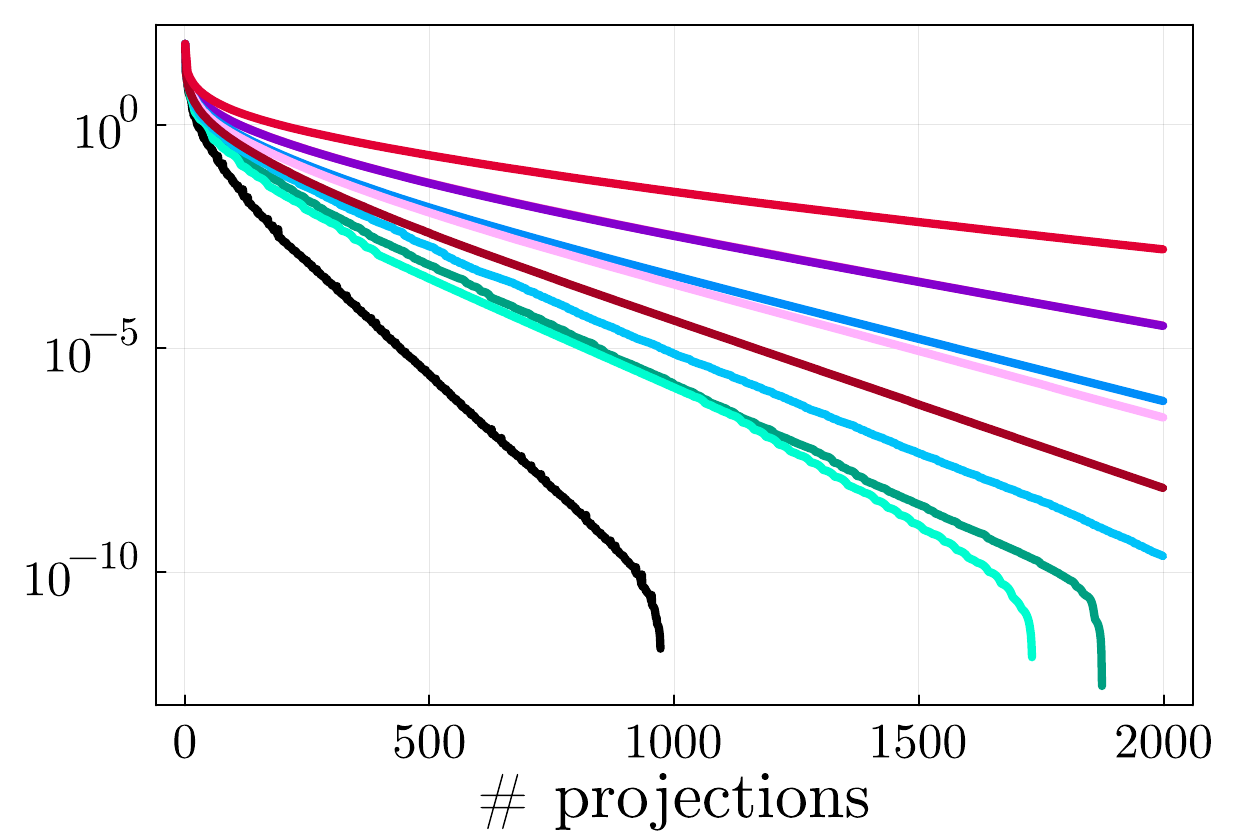}
    \caption{$m=50$, $n=500$}\label{fig:mlc-2}
  \end{subfigure}
   
  \caption{{Minimal length piecewise-linear curve, problem~\eqref{eq:mlc}}}
  \label{fig:mlc}
\end{figure}

\paragraph{Nonnegative matrix factorization.}
We want to solve the matrix factorization problem subject to nonnegative constraints:
\begin{equation}\label{eq:nmf}
  \min_{U,V\in \R^{n\times r}_{+}} f(U,V) = \frac 12 \n{UV^\top - A}^2_{F},
\end{equation}
where $A$ is a given $n$-by-$n$ low-rank matrix. Although nonconvex, this
problem is famously well-tackled by first-order methods. In each iteration, the
gradient $\nabla f(x)$ involves 3 matrix-matrix multiplications, whereas
evaluating the objective $f(x)$ only requires 1. Note that for the last
iteration of the linesearch, the computed matrix product can be reused to
compute the gradient for the next iteration.

We created matrix $A$ by
multiplying matrices $B$ and $C^\top$, where $B$ and $C$ are $n$-by-$r$ matrices
with entries sampled from a normal distribution. Negative entries in both
matrices $B$ and $C$ were then set to zero. The results are presented in
\Cref{fig:nmf}, where the number of gradients  roughly means the number of 3 matrix-matrix multiplications.

\begin{figure}[ht]
  \mylegend
  \centering
  \begin{subfigure}[b]{0.45\linewidth}
    \centering
    \includegraphics[width=\linewidth]{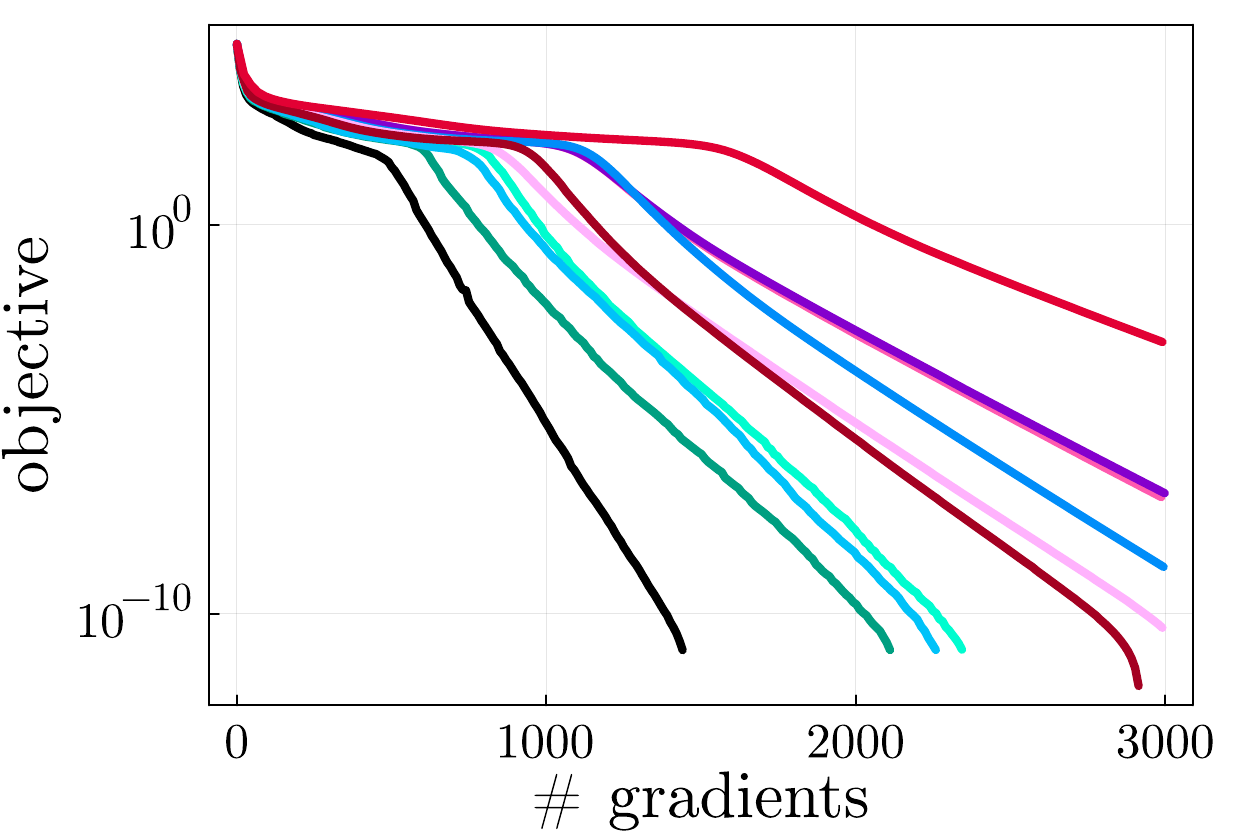}
    \caption{$n=100$, $r=20$}\label{fig:nmf-1}
  \end{subfigure}
  \begin{subfigure}[b]{0.45\linewidth}
    \centering
    \includegraphics[width=\linewidth]{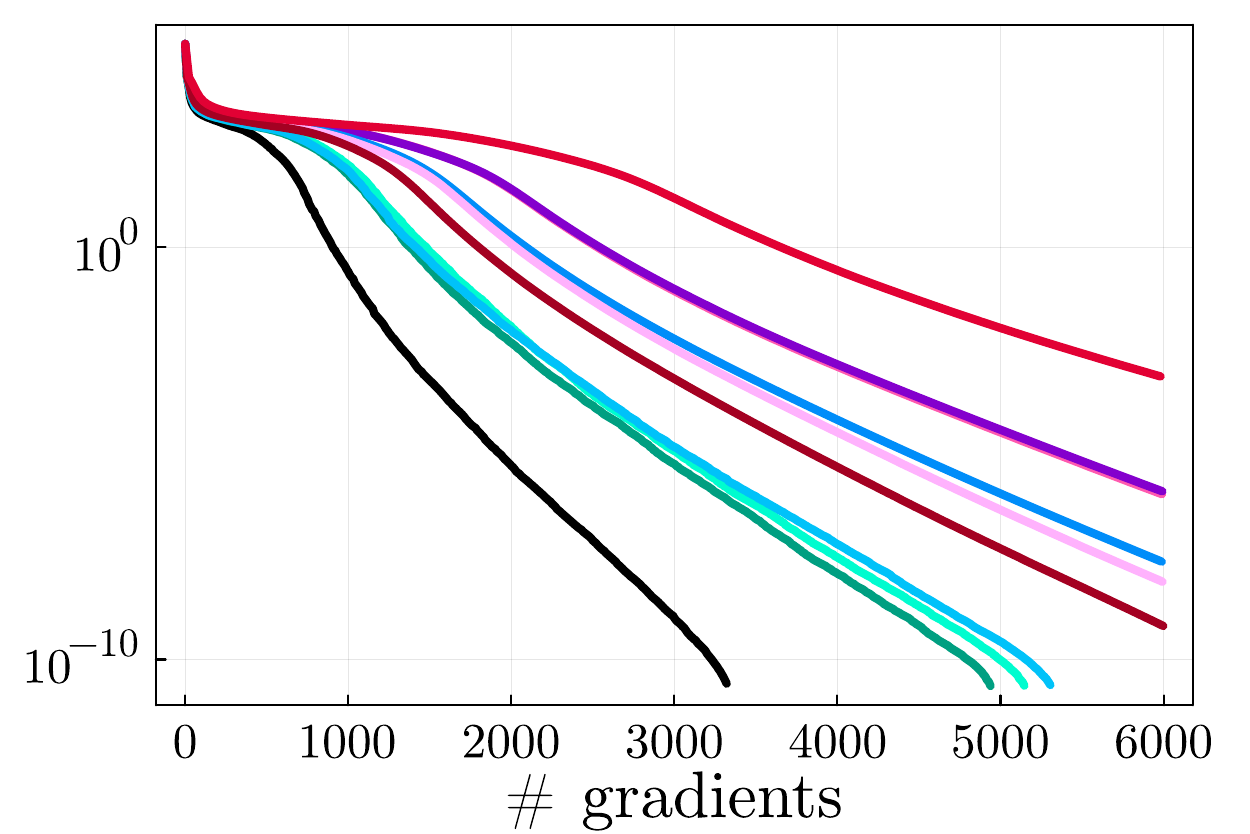}
    \caption{$n=100$, $r=30$}\label{fig:nmf-2}
  \end{subfigure}
   
  \caption{{Nonnegative matrix factorization, problem~\eqref{eq:nmf}}}
  \label{fig:nmf}
\end{figure}

\paragraph{Dual of the entropy maximization problem.} Consider the entropy
maximization problem subject to linear constraints
\begin{equation}
  \label{eq:maxentropy}
    \min \sum_{i=1}^nx_i \log x_i \quad \text{subject to $Ax \leq b$,\quad
     $\sum_{i=1}^n x_i=1$, and $x_i>0$, }
 \end{equation}
 where $A\in \R^{m\times n}$.
Its dual problem is given by
\begin{equation}
  \label{eq:dual_maxent}
  \min_{\la\in \R^m_{+}, \mu\in \R}   e^{-\mu-1}\sum_{i=1}^ne^{-a_i^\top \la}+\lr{b, \la} + \mu,
\end{equation}
where $a_i\in \R^m$ is the $i$-th column of $A$ (the derivation is provided in
\cite[Chapter 5.1.6]{boyd2004convex}).
\begin{figure}[ht]
  \mylegend
  \centering
  \begin{subfigure}[b]{0.45\linewidth}
    \centering
    \includegraphics[width=\linewidth]{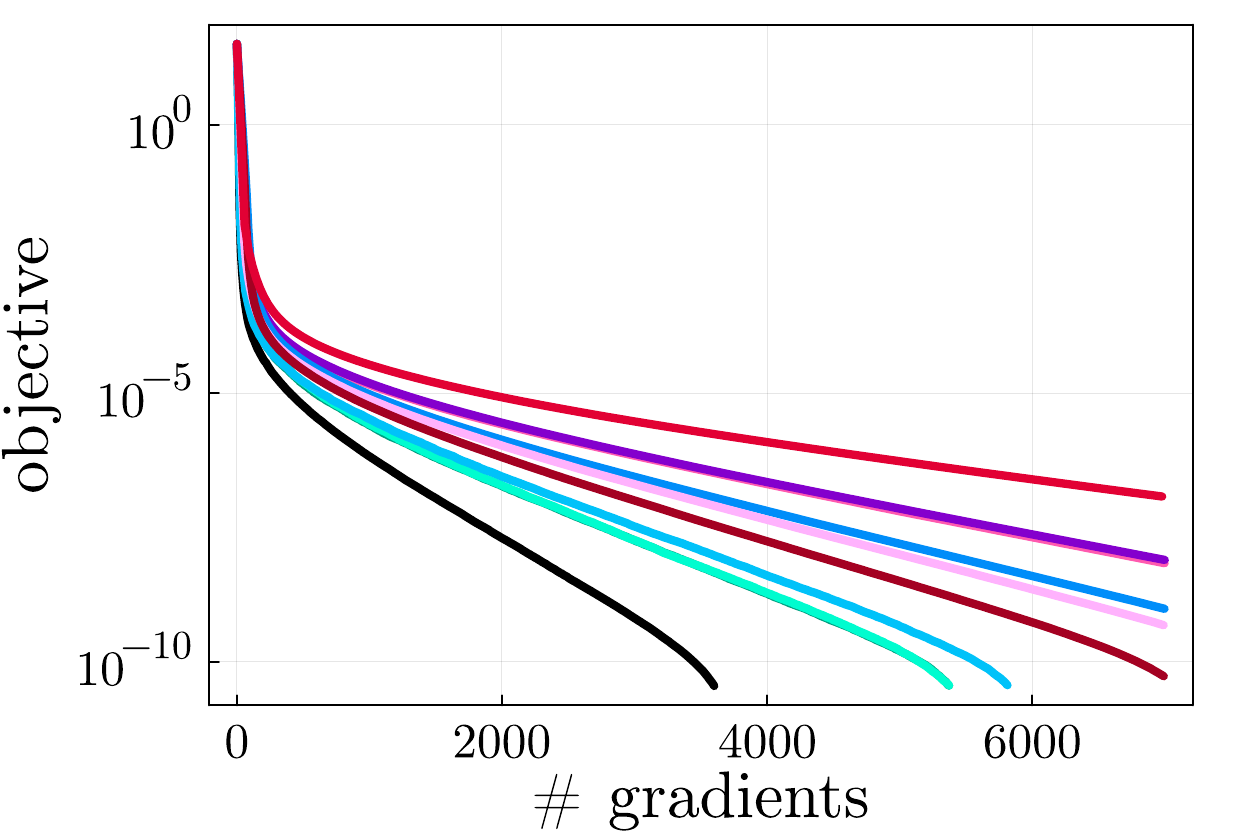}
    \caption{$m=500$, $n=100$}\label{fig:dual_maxent-1}
  \end{subfigure}
  \begin{subfigure}[b]{0.45\linewidth}
    \centering
    \includegraphics[width=\linewidth]{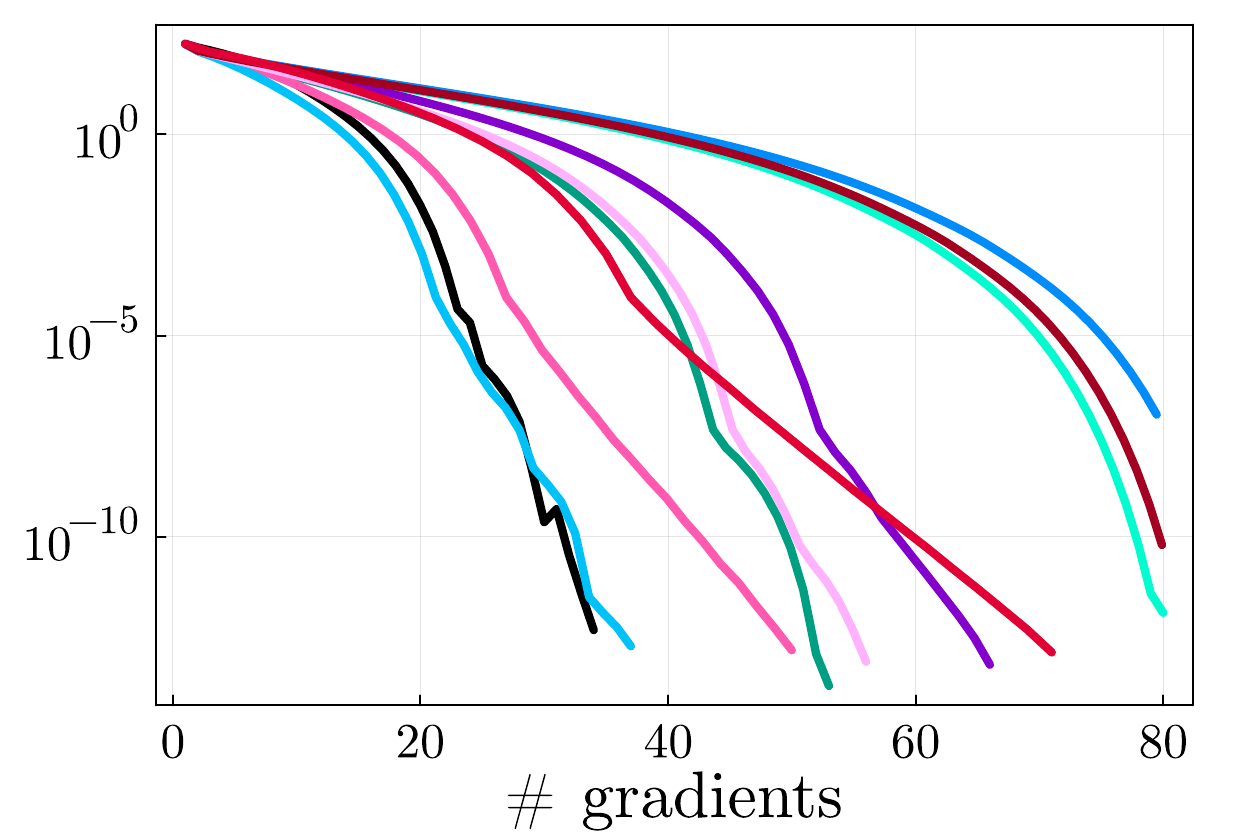}
    \caption{$m=100$, $n=500$}\label{fig:dual_maxent-2}
  \end{subfigure}
   
  \caption{{Dual of the entropy maximization, problem~\eqref{eq:dual_maxent}}}
  \label{fig:dual_maxent}
\end{figure}
 
It is the latter problem~\eqref{eq:dual_maxent} that we solved. We generated
$m$-by-$n$ matrix $A$ with entries sampled from a normal distribution. Then we
generated a random $w\in \R^n$ from the unit simplex and set $b = Aw$. Each
gradient requires two matrix-vector multiplication, while the objective only one
(and, as before, the last one can be reused for the next gradient). The results
are presented in~\Cref{fig:dual_maxent}, where the number of gradients roughly
means the number of 2 matrix-vector multiplications.

\paragraph{Conclusion.}
Based on our preliminary experiments, it is evident that AdProxGD indeed performs better. 
To our surprise, a few specific pairs $(r,s)$ consistently outperform the rest among
ProxGD with linesearch. We are not aware of any theoretical finding that would
confirm this evidence.

\section*{\hypertarget{appendix}{Appendix}}

\begin{lemma}\label{prop:bad_f}
The function $f$ defined in~\eqref{eq:counter} satisfies the following properties:
        \begin{enumerate}
                \item $f$ is convex.
                \item $f'$ is $L$-Lipschitz with $L=1$.
                \item $f$ is locally strongly convex, i.e., for any bounded set $\mathcal{X}$ there exists a constant $\mu_{\mathcal{X}}>0$ such that $|f'(x)-f'(y)|\geq \mu_{\mathcal{X}}|x-y|$ for any $x,y\in \mathcal{X}$.
                \item $|f'(x)|\leq G$ with $G=2$.
                \item $f$ is $2$-Lipschitz.
        \end{enumerate}
      \end{lemma}
      \begin{proof} 
First, let us find $f'$ and $f''$:
\begin{align*}
        &f'(x)
        = \begin{cases}
                x,& x\in [-1, 1] \\
                \frac{ax}{1+|x|}, & x\not\in [-1, 1]
        \end{cases},\qquad
        &f''(x)
        = \begin{cases}
                1,& x\in (-1, 1) \\
                \frac{a}{(1+|x|)^2}, & x\not\in [-1, 1]
        \end{cases},
\end{align*}
so indeed  $a=2$ and $b=2\log 2-\frac{3}{2}$. Convexity of $f$ follows from the fact that $f''(x)>0$ for any
  $x$. Lipschitzness of $f'$ follows directly from the bound $f''(x)\leq 1$ for
  all $x$. Similarly, local strong convexity follows from the bound $f''(x)\ge
  \frac{1}{\max_{z\in \mathcal{X}}(1+|z|)^2}\eqqcolon\mu_\cX$ for any $x\in
  \cX$. Finally, the last two properties trivially follow from the expression for $f'(x)$.
\end{proof}
\begin{proof}[Proof of \Cref{prop:counter}]\label{prooflemma1}
 Let us choose $x^0=r + 2 $  with a sufficiently large $r > 6c$. 
  This readily implies that
\[x^1 = x^0 - \frac{2x^0}{1+x^0} = \frac{x^0(x^0-1)}{x^0+1}>x^0-2 = r.\]
  Our goal is to show that the iterates follow a very specific pattern. Namely, we prove that for all $k\geq 0$,
\[\sign(x^{2k}) = \sign(x^{2k+1}), \quad \sign(x^{2k+2})\neq \sign(x^{2k}), \quad |x^{2k+2}|>2|x^{2k+1}| > |x^{2k}|.\]
If this condition holds true, then the sequence $(x^k)$ must be divergent.

  First, observe that if $|x^k|, |x^{k-1}|\geq r$ and $\sign(x^k)=\sign(x^{k-1})$, then the smoothness estimate admits a simple expression:
        \begin{align*}
                L_k
                &= \frac{|f'(x^k) - f'(x^{k-1})|}{|x^k - x^{k-1}|}
                = \frac{2\left|\frac{x^k}{1+|x^k|}-\frac{x^{k-1}}{1+|x^{k-1}|} \right|}{|x^k - x^{k-1}|}
                = \frac{2\left|\frac{|x^k|}{1+|x^k|}-\frac{|x^{k-1}|}{1+|x^{k-1}|} \right|}{||x^k| - |x^{k-1}||}\\
                &= \frac{2}{(1+|x^{k}|)(1+|x^{k-1}|)}< \frac{2}{r(1+|x^{k}|)}.
        \end{align*}
Therefore, in that case $\ss_k|f'(x^k)|>\frac{r(1+|x^{k}|)}{2c}|f'(x^k)|=\frac{r|x^k|}{c}>3|x^k|$. Since $\sign(f'(x^k))=\sign(x^k)$, it implies that $|x^{k+1}|>2|x^k|$ and $\sign(x^{k+1})\neq \sign(x^k)$.

        Next, if $|x^k|, |x^{k-1}|\geq r > 3$ with $|x^k|\geq 2|x^{k-1}|$ and $\sign(x^k)\neq \sign(x^{k-1})$, then we have $\frac{2|x^k|}{1+|x^k|}>\frac 3 2$ and
        \begin{align*}
                L_k
                &= \frac{|f'(x^k) - f'(x^{k-1})|}{|x^k - x^{k-1}|}
                = \frac{2\left|\frac{x^k}{1+|x^k|}-\frac{x^{k-1}}{1+|x^{k-1}|} \right|}{|x^k - x^{k-1}|}
                = \frac{2\left(\frac{|x^k|}{1+|x^k|}+\frac{|x^{k-1}|}{1+|x^{k-1}|} \right)}{|x^k| + |x^{k-1}|} \\
                &> \frac{3}{|x^k| + |x^{k-1}|}
                > \frac{3}{|x^k| + \frac{1}{2}|x^{k}|}
                \geq  \frac{2}{|x^k|}.
        \end{align*}
        This implies $\ss_k < \frac{|x^k|}{2c}\leq \frac{|x^k|}{2}$. Since
        $\sign(f'(x^k)) = \sign(x^k)$ and $\frac{\ss_k}{1+|x^k|}\leq \frac{|x^k|}{2(1+|x^k|)}<\frac 1 2 $, we conclude that $\sign(x^{k+1}) = \sign(x^k)$ and
        \[|x^{k+1}|=\left|x^{k}-\ss_k\frac{x^{k}}{1+|x^k|}\right| = |x^k| \left(1 - \frac{\ss_k}{1+|x^k|}\right) >  \frac{1}{2}|x^k|.\]
As $x^0$ and $x^1$ satisfy the first case, by induction we deduce that all iterates $(x^k)$ follow the described pattern.
\end{proof}

\paragraph{Acknowledgments.} The authors would like to thank Puya Latafat, who found a subtle error in the convergence proof of $(x^k)$ in~\Cref{th:main-2} in the first version of this manuscript.

\printbibliography

\end{document}